\documentclass[10pt,reqno]{amsart}
\usepackage[utf8]{inputenc}

\title[Ancient Rescaled MCF]{On the Existence and Uniqueness of Ancient Rescaled Mean Curvature Flows}
\author{Letian Chen}
\address{Department of Mathematics, Johns Hopkins University, 3400 N. Charles Street, Baltimore, MD 21218}
\email{lchen155@jhu.edu}
\date{December 21, 2022}

\usepackage[numbers]{natbib}
\usepackage{amsmath}
\makeatletter
\renewcommand*\env@matrix[1][*\c@MaxMatrixCols c]{%
	\hskip -\arraycolsep
	\let\@ifnextchar\new@ifnextchar
	\array{#1}}
\makeatother
\usepackage{amsfonts}
\usepackage{amssymb}
\usepackage{mathrsfs}
\usepackage{mathtools}
\usepackage{enumitem}
\usepackage{url}
\usepackage[pagebackref=true]{hyperref}

\usepackage[noabbrev,capitalize,nameinlink]{cleveref}
\usepackage{flafter}

\usepackage{amsthm}
\newtheorem{thm}{Theorem}[section]
\newtheorem{lem}[thm]{Lemma}
\newtheorem{prop}[thm]{Proposition}
\newtheorem{cor}[thm]{Corollary}

\theoremstyle{definition}

\theoremstyle{remark}
\newtheorem*{rem}{Remark}

\newcommand\abs[1]{\left|#1\right|}
\newcommand\norm[1]{\left\lVert#1\right\rVert}
\newcommand\inner[1]{\left \langle #1\right \rangle}

\newcommand{\nN}[0]{\mathbf{n}}
\newcommand{\xX}[0]{\mathbf{x}}

\newcommand{\hh}[0]{\mathbf{h}}
\newcommand{\hH}[0]{\mathbf{H}}
\newcommand{\fF}[0]{\mathbf{f}}
\newcommand\loc{\mathrm{loc}}
\newcommand\rel{\mathrm{rel}}

\newcommand{\eps}{\varepsilon}
\newcommand{\R}{\mathbb{R}}

\newcommand\wt{e^{\frac{\abs{\xX}^2}{4}}}

\DeclareMathOperator{\Div}{div}
\DeclareMathOperator{\tr}{tr}
\DeclareMathOperator{\supp}{supp}

\numberwithin{equation}{section}

\crefformat{equation}{(#2#1#3)}
\crefformat{enumi}{(#2#1#3)}

\begin{document}
\begin{abstract}
	We show existence of ancient solutions to the rescaled mean curvature flow starting from a given asymptotically conical self-expander. These are examples of mean curvature flows coming out of cones that are not self-similar. We also show a strong uniqueness theorem when the cone is generic and use it to classify mean curvature flows coming out of generic cones of small entropy in low dimensions.
\end{abstract}
\maketitle
\section{Introduction}
A family of properly embedded hypersurface $\{\tilde{\Sigma}_s\}_{s \in I} \subset \mathbb{R}^{n+1}$ is a \textit{rescaled mean curvature flow} (RMCF) if it satisfies the following equation: 
\begin{align}
	\label{rescaled-mcf-equation}
	\left(\frac{\partial \xX}{\partial s}\right)^\perp = \hH_{\tilde{\Sigma}_s} - \frac{\xX^\perp}{2}.
\end{align}
Here $\hH_{\tilde{\Sigma}_s}$ is the mean curvature vector of $\tilde{\Sigma}_s$ and $\perp$ denotes the normal component. The RMCF equation comes from the \textit{mean curvature flow} (MCF) equation
\begin{align}
\label{mcf-equation}
	\left(\frac{\partial \xX}{\partial t}\right)^\perp = \hH_{\Sigma_t}.
\end{align}
via the change of variables 
\begin{align}
\label{change-of-variable}
	s = \log t \text{ and } \tilde{\Sigma}_s = t^{-1/2}\Sigma_t.
\end{align} 
We emphasize that there are (at least) two different flows that are referred to as the RMCF, namely \cref{rescaled-mcf-equation} and \cref{rescaled-mcf-equation} with its right hand side replaced by $\hH_{\tilde{\Sigma}_s} + \frac{\xX^\perp}{2}$. The latter flow arises naturally in the study of self-shrinkers. In this paper, however, we are only concerned with the former equation.\par 
The static solutions to \cref{rescaled-mcf-equation} are known as the \textit{self-expanders}; that is, smooth embedded hypersurfaces $\Sigma \subset \R^{n+1}$ satisfying
\begin{align}
	\hH_\Sigma = \frac{\xX^\perp}{2}. \label{expander-equation}
\end{align}
Equivalently, $\Sigma$ is a self-expander if and only if $\{\sqrt{t}\Sigma\}_{t \in (0,\infty)}$ is a solution to the MCF equation \cref{mcf-equation}. Self-expanders are modeled on MCFs coming out of (hyper)cones and serve as potential continuations of MCFs past conical singularities. Given a smooth cone $\mathcal{C} \subset \R^{n+1}$, a hypersurface $\Sigma$ is \textit{$C^{k,\alpha}$-asymptotic} to $\mathcal{C}$ if 
\begin{align*}
	\lim_{\rho \to 0^+} \rho \Sigma = \mathcal{C} \text{ in } C^{k,\alpha}_{\loc}(\mathbb{R}^{n+1} \setminus \{0\}).
\end{align*} 
Asymptotically conical self-expanders can be thought naturally as MCFs coming out of their asymptotic cones in the sense that 
\begin{align*}
	\lim_{t \to 0} \mathcal{H}^n \llcorner (\sqrt{t}\Sigma) = \mathcal{H}^n \llcorner \mathcal{C}.
\end{align*} \par
As cones are singular at the origin, their evolution by MCF needs not be unique. In particular, there may be multiple distinct smooth self-expanders asymptotic to the same cone. This nonuniqueness was first confirmed by Angenent--Ilmanen--Chopp \cite{AngenentIlmanenChopp}, who showed numerically that there exist at least two self-expanders coming out of a sufficiently wide rotationally symmetric double cone. The argument was later made rigorous by \cite{Helmensdorfer} (cf. \cite[Theorem 1.2]{BWIntegerDegree}). It turns out that there exist MCFs coming out of $\mathcal{C}$ that are not self-similar, provided there are more than one smooth self-expanders asymptotic to $\mathcal{C}$. One class of examples of non-self-similar MCFs is given by \textit{monotone Morse flow lines} connecting an unstable self-expander and a stable self-expander. These solutions are constructed in \cite{BCW} using ideas from \cite{BWTopologicalUniqueness} and turn out to possess more regularity than a general MCF coming out of $\mathcal{C}$. We also refer to \cite[Section 5]{ChenII} for a (simpler) low entropy version of the construction, and to \cite[Section 7]{CCMSGeneric} for a related construction for an asymptotically conical self-shrinker. \par 
Monotone Morse flow lines are a special case of \textit{Morse flow lines}, which are RMCFs connecting two (smooth) self-expanders. More precisely, $\tilde{\Sigma}_s$ is a Morse flow line if there exist smooth self-expanders $\Sigma$ and $\Gamma$ such that
\begin{align*}
	\lim_{s \to -\infty} \mathcal{H}^n \llcorner \Sigma_s = \mathcal{H}^n \llcorner \Sigma \text{ and } \lim_{s \to \infty} \mathcal{H}^n \llcorner \Sigma_s = \mathcal{H}^n \llcorner \Gamma. 
\end{align*}
In particular, a Morse flow line is an ancient solution of \cref{rescaled-mcf-equation} (that is, flows that are defined on $(-\infty, S)$). We shall also distinguish between the notions of a \textit{smooth} Morse flow line, which is smooth on $(-\infty, \infty)$, and a \textit{singular} Morse flow line (which must be described using a suitable notion of weak MCF - see \cref{preliminaries}). Even though the self-expanders $\Sigma$ and $\Gamma$ are smooth, the flow lines are not going to be smooth in general (see \cite[Proposition 5.7]{ChenII} for an explicit example). The biggest difficulty in the study of singular flow lines is the possibility that they may not be canonically continued past a singularity due to the aforementioned nonuniqueness phenomenon with singular initial data. In \cite{BCW}, we constructed a natural singular Morse flow line when the flow is expander mean convex (see also \cref{existence-flow-line}). \par 
In this article, we are interested in another special class of MCFs coming out of $\mathcal{C}$. Fix a smooth self-expander $\Sigma$ asymptotic to $\mathcal{C}$, we consider RMCFs $\{\tilde{\Sigma}_s\}$ on $(-\infty,0)$ such that 
\begin{align}
\label{asymptotic-conical-cond}
	\lim_{s \to -\infty}  \tilde{\Sigma}_{s} =  \Sigma \text{ in } C^{\infty}_{\loc}(\R^{n+1}).
\end{align}
Such a flow will be called a \textit{tame ancient} RMCF (in general there could exist very wild ancient RMCFs that do not have a well-defined backward limit or even arise as MCFs out of a cone). It is easy to see that a tame ancient RMCF gives rise to a MCF coming out of $\mathcal{C}$ in view of \cref{change-of-variable}. The converse is, in general, not true - see \cref{cone-section} for some related discussion. \par 
Our first result is an existence theorem for tame ancient RMCFs when $\Sigma$ is unstable. In particular, this gives examples of many non-self-similar flows coming out of a cone with more than one asymptotically conical self-expanders (see \cref{preliminaries} for terminologies). 
\begin{thm}
\label{main-theorem}
	Let $\mathcal{C} \subset \R^{n+1}$ be a smooth cone and suppose $\Sigma$ is an unstable smooth self-expander of index $I$ asymptotic to $\mathcal{C}$, then there exists an $I$-parameter family of distinct tame ancient RMCFs $\{\tilde{\Sigma}_s\}_{s \in (-\infty,0]}$ such that 
	\begin{align*}
		\lim_{s \to -\infty} \tilde{\Sigma}_s = \Sigma \text{ in } C^{\infty}_{\loc}(\R^{n+1}).
	\end{align*}
	In particular, there exists an $I$-parameter family of smooth MCFs coming out of $\mathcal{C}$. 
\end{thm}
\cref{main-theorem} is a non-compact analogue of the general existence theorem of ancient gradient flows of certain elliptic functionals proven by Choi--Mantoulidis \cite{ChoiMantoulidis}. The method is also adapted to self-shrinkers \cite[Section 6]{CCMSGeneric}. We also mention the recent work of Sun--Wang \cite{SunWang} where they constructed an $I$-parameter family of translators asymptotic to $\Sigma \times \R$ for a self-shrinker $\Sigma$ of index $I$. The space of solutions will correspond to the space of eigenfunctions of negative eigenvalues. \par 
If, in addition, the cone $\mathcal{C}$ is generic, in the sense that any self-expander asymptotic to $\mathcal{C}$ has nontrivial Jacobi field, we also have the following strong uniqueness result:
\begin{thm}
	\label{main-uniqueness-theorem}
	Suppose $\mathcal{C} \subset \mathbb{R}^{n+1}$ is a generic smooth cone. Suppose $\tilde{\mathcal{M}} = \{\tilde{\Sigma}_s\}_{s \in (-\infty,0)}$ is a smooth RMCF such that
	\begin{align}
		\label{back-convergence}
		\lim_{s \to -\infty}  \tilde{\Sigma}_{s} =  \Sigma \text{ in } C^{\infty}_{\loc}(\R^{n+1})
	\end{align}
	for some smooth self-expander $\Sigma$ asymptotic to $\mathcal{C}$. Then $\Sigma$ is unstable and $\tilde{\mathcal{M}}$ is either the static flow of $\Sigma$ or coincides with one of the tame ancient RMCFs constructed from \cref{main-theorem}.
\end{thm}
The proof of \cref{main-uniqueness-theorem} relies on the Merle--Zaag ODE lemma \cref{ode-lemma}, which has found much recent success in various classification problems for ancient geometric flows, see for example \cite{ADS}, \cite{ChoiMantoulidis}, \cite{CHHAncient}, \cite{CHHW}, and \cite{DuHaslhofer}. The genericity of $\mathcal{C}$ is used to rule out the scenario in \cref{ode-lemma} where the neutral mode (i.e., nontrivial Jacobi fields) become dominant as $s \to -\infty$. We suspect that \cref{main-uniqueness-theorem} is not true without the genericity assumption. For instance, if there exists a 1-parameter family of weakly stable self-expanders $\{\Sigma^a\}_{a \in (0,1)}$ asymptotic to $\mathcal{C}$, it is possible that there exists a Morse flow line between $\Sigma^{a_1}$ and $\Sigma^{a_2}$, which will violate both conclusions of the theorem. \par
Since every Morse flow line is a tame ancient RMCF, we immediately have the following consequence of \cref{main-uniqueness-theorem}
\begin{cor}
\label{main-cor}
Every Morse flow line starting from a smooth self-expander $\Sigma$ asymptotic to a generic cone agrees with one of the tame ancient RMCF starting from $\Sigma$ constructed in \cref{main-theorem}.
\end{cor}
The above corollary gives a reasonably complete picture of the flow lines asymptotic to a generic cone. We point out that the converse does not hold, as being a tame ancient flow tells nothing about the forward limit $\Gamma$. In fact, the forward limit $\Gamma$ appears to be much more complicated to analyze. One particular case in which we can deduce regularity of $\Gamma$ is when $\mathcal{C}$ is generic and $\Sigma$ has index 1. Indeed, by \cref{main-cor}, any Morse flow line belongs to the 1-parameter family of tame ancient RMCF starting from $\Sigma$ corresponding to its first eigenfunction. This implies that the flow actually lies on one side of $\Sigma$, and thus $\Gamma$ is stable (see \cite[Proposition 5.1]{BWTopologicalUniqueness}). In particular, $\Gamma$ has a singular set of codimension at least 7 and is smooth in low dimensions.  \par 
Recall that a tame ancient RMCF starting from an asymptotically conical self-expander $\Sigma$ gives rise to a MCF starting from its asymptotic cone. It is therefore natural to ask when the same conclusion of \cref{main-uniqueness-theorem} holds for a general MCF coming out of the cone.
This problem turns out to be more subtle as a MCF coming out of the cone is not necessarily a tame ancient RMCF. The failure of this correspondence is essentially the failure of the uniqueness of tangent flows, which, in turn, is due to the lack of regularity of the backward limit. To elaborate, given a MCF, $\mathcal{M}$, coming out of $\mathcal{C}$, the corresponding RMCF, $\tilde{\mathcal{M}}$ (via \cref{change-of-variable}), only satisfies 
\begin{align*}
	\lim_{i \to \infty} \tilde{\mu}_{s_i} = \mathcal{H}^n \llcorner \Sigma,
\end{align*}
for some subsequence $s_i \to -\infty$ and self-expander $\Sigma$ (not necessarily smooth - see \cref{backward-limit}). To upgrade the subsequential convergence to full convergence, the standard tool is the \L ojasiewicz--Simon inequality \cite{SimonParabolic} (cf. \cite{Schulze}, \cite{ChodoshSchulze} and \cite{CMUniqueness} for various applications in MCF). In our case, we will prove a \L ojasiewicz inequality when the cone is generic.
\begin{thm}[\L ojasiewicz inequality for generic cones, \cref{lojasiewicz-relative-entropy}]
	Let $\Sigma$ be a smooth self-expander asymptotic to a generic cone $\mathcal{C}$. There is $\eps = \eps(\Sigma)$ such that the following holds: suppose $v \in C^{2,\alpha} \cap W^2(\Sigma)$ satisfies $\norm{v}_{C^{2,\alpha}}  < \eps$, then 
	\begin{align*}
		C\norm{\mathcal{N}_\Sigma(v)}_{W} \ge \abs{E_{\rel}^*[\Sigma_v,\Sigma]}^{1/2},
	\end{align*}
	where $\mathcal{N}_\Sigma$ is the Euler-Lagrange operator associated to the relative expander entropy $E_{\rel}$.
\end{thm}
\begin{rem}
	Without the genericity assumption, one instead expects a constant $\gamma \in (0,1)$ such that 
	\begin{align}
		\label{better-lojasiewicz}
		C\norm{\mathcal{N}_\Sigma(v)}_{W} \ge \abs{E_{\rel}[\Sigma_v,\Sigma]}^{1- \gamma/2},
	\end{align}
	which is the usual \L ojasiewicz inequality for $E_{\rel}$ proved by Park--Wang \cite{ParkWang} (here, for technical reasons, the two relative expander entropies $E_{\rel}$ and $E_{\rel}^*$ are defined differently - see \cref{uniqueness-section}). This is, of course, also enough to deduce the uniqueness of tangent. However, as \cref{main-uniqueness-theorem} requires genericity anyways, we have opted to prove the inequality in the generic case separately. We also emphasize that the proof of the theorem makes no use of the classical \L ojasiewicz--Simon inequality from \cite{SimonParabolic} and, therefore, also does not rely on the analyticity of the relative expander functional (which is indeed true). See \cite{FeehanMaridakis} for some related discussion on analyticity vs. Morse-Bott conditions. \par 
\end{rem}
The smoothness of $\Sigma$ is, unfortunately, essential in the above theorem. For this reason we only have a partial converse in the MCF case:
\begin{cor}
\label{main-cor-mcf}
	Suppose $2 \le n \le 6$ and $\mathcal{C} \subset \R^{n+1}$ is a generic cone. Suppose $\mathcal{M} = \{\Sigma_t\}_{t \in (0,T)}$ is a smooth MCF such that the corresponding RMCF satisfies 
	\begin{align}
	\label{subsequent-convergence}
		\lim_{i \to \infty} \mathcal{H}^n \llcorner \tilde{\Sigma}_{s_i} = \mathcal{H}^n \llcorner \Sigma
	\end{align}
	for some sequence $s_i \to -\infty$ and some \textbf{smooth} self-expander $\Sigma$ asymptotic to $\mathcal{C}$, then $\mathcal{M}$ agrees with one of the tame ancient RMCF starting from $\Sigma$ constructed in \cref{main-theorem}.
\end{cor}
\begin{rem}
	Here the dimension restriction is due to the trapping phenomenon, i.e. we need the flow $\mathcal{M}$ to be trapped between two asymptotically conical self-expanders, which is only known in low dimensions \cite[Proposition 8.21]{CCMSGeneric}.
\end{rem}
\begin{rem}
	The regularity of the subsequential limit is in general hard to check, but there are natural geometric conditions (i.e. entropy bounds) that force the limit to be smooth. In these cases, \cref{main-cor-mcf} results in a classification of MCFs coming out of generic cones of low entropy - see \cref{cor-r3} and \cref{cor-r4}.
\end{rem}
The rest of the article is organized as follows. In \cref{existence-section} we prove \cref{main-theorem}. In \cref{uniqueness-section} we first define a version of the relative expander entropy which gives the equation \cref{rescaled-mcf-equation} a gradient flow structure, and then carry out the spectral analysis to prove \cref{main-uniqueness-theorem}. In \cref{cone-section} we study MCFs coming out of cones and prove the \L ojasiewicz inequality \cref{lojasiewicz-relative-entropy}. We then prove a uniqueness of tangent flows result for generic cones and combine it with \cref{main-uniqueness-theorem} to show \cref{main-cor-mcf}. Finally, we give some applications of \cref{main-cor-mcf} in the low entropy case.

\subsection*{Acknowledgement}
I would like to thank my advisor, Jacob Bernstein, for his constant encouragement and numerous suggestions to the article. I thank Jiewon Park, Lu Wang and Junfu Yao for helpful conversations. 

\section{Preliminaries}
\label{preliminaries}
\subsection{Notations}	
By a \textit{(hyper)cone} we mean a set $\mathcal{C} \subset \mathbb{R}^{n+1}$ that is invariant under dilations, i.e. $\rho \mathcal{C} = \mathcal{C}$ for all $\rho > 0$. We say $\mathcal{C}$ is smooth if it is smooth away from the origin. We say a hypersurface $\Sigma$ is asymptotic to $\mathcal{C}$ if it is smoothly asymptotic to $\mathcal{C}$, in view of \cref{asymptotic-conical-cond}. \par
Given a (smooth) hypersurface $\Sigma \subset \mathbb{R}^{n+1}$ and any function $v: \Sigma \to \mathbb{R}$, let $\Sigma_v$ be the normal graph of $v$ over $\Sigma$, i.e. $\Sigma_v = \fF_v(\Sigma)$ where 
\begin{align*}
	\fF_v(p)  = \xX(p) + v(p)\nN_\Sigma(p), \; \; p \in \Sigma.
\end{align*} \par 
The time variable $t$ is always used in the unrescaled setting (i.e. MCF), and $s$ is always used in the rescaled setting (i.e. RMCF). 
\subsection{Integral Brakke flow}
\label{mcf-preliminaries}
We will be using notations from geometric measure theory below, and we refer to \cite{SimonBook} and \cite{Ilmanen} for the relevant definitions. Let $V(\mu)$ denote the associated integral varifold for an integral $n$-rectifiable Radon measure $\mu$. Let $\hH$ be the generalized mean curvature vector of $V(\mu)$ given by the first variation formula: for any compactly supported $C^1$ vector field,
\begin{align*}
	\int \Div_{V(\mu)} X d\mu = - \int \hH \cdot X d\mu.
\end{align*}
Given an open set $U \subset \mathbb{R}^{n+1}$, by an \textit{integral $n$-Brakke flow} in $U$ we mean a family of integral $n$-rectifiable Radon measures $\mathcal{M} = \{\mu_t\}_{t \in I}$ such that:
\begin{enumerate}[label = (\alph*)]
	\item For a.e. $t \in I$, $V(\mu)$ has locally bounded first variation and its generalized mean curvature vector $H$ is orthogonal to the approximating tangent space of $V(\mu)$ $x$-a.e.
	\item For any bounded interval $[a,b] \subset I$ and compact set $K \subset U$,
	\begin{align*}
		\int_a^b \int_K (1 + \abs{\hH}^2) d\mu_tdt < \infty.
	\end{align*}
	\item For $[a,b] \subset I$ and every $\phi \in C_c^1(U \times [a,b], \mathbb{R}^+)$,
	\begin{align*}
		\int \phi d\mu_b - \int \phi d\mu_a \le \int_a^b \int \left(-\phi\abs{\hH}^2 + \hH \cdot \nabla \phi + \frac{\partial \phi}{\partial t}\right)d\mu_tdt.
	\end{align*}
\end{enumerate}
Throughout the article, a Brakke flow is assumed to be integral, cyclic, and unit-regular (see eg. \cite{WhiteCurrentsChains}). This ensures that there is no sudden loss of mass along the flow. Any Brakke flow produced by elliptic regularization \cite{Ilmanen} will automatically have these properties. \par 
Under the rescaling $s = \log t$ and $\tilde{\mu}_s = \mu_t \circ t^{\frac{1}{2}}$, a Brakke flow becomes a rescaled Brakke flow, which satisfies the same conditions as above, with condition (c) replaced by 
\begin{enumerate}
	\item [(c')] For $[a,b] \subset I$ and every $\phi \in C_c^1(U \times [a,b], \mathbb{R}^+)$,
	\begin{align*}
		\int \phi d\mu_b - \int \phi d\mu_a \le \int_a^b \int \left(-\phi H\cdot(-\frac{\xX^\perp}{2} + \hH) + (-\frac{\xX^\perp}{2} + \hH) \cdot \nabla \phi + \frac{\partial \phi}{\partial t}\right)d\mu_tdt.
	\end{align*}
\end{enumerate}
Of course, a smooth rescaled Brakke flow satisfies the RMCF equation, \cref{rescaled-mcf-equation}. 
\subsection{Function Spaces}
Let $\Sigma$ be a smooth hypersurface. For $\alpha \in (0,1)$ we define the H\"{o}lder seminorm
\begin{align*}
	[f]_{\alpha; \Sigma} = \sup_{\substack{p \ne q \in \Sigma \\ q \in B_\delta^\Sigma(p)}} \frac{\abs{f(p) - f(q)}}{d_\Sigma(p,q)^{\alpha}},
\end{align*}
where $\delta$ is such that $B_\delta^\Sigma(p)$ is strictly geodesically convex. For $k$ an integer, define $C^{k,\alpha}(\Sigma)$ to be the H\"{o}lder space with norm 
\begin{align*}
	\norm{f}_{C^{k,\alpha}(\Sigma)} = \sum_{i=0}^k \sup_{\Sigma} \abs{\nabla_\Sigma^i f}  + [\nabla_\Sigma^k f]_{\alpha;\Sigma }.
\end{align*}
The parabolic H\"{o}lder seminorm is defined similarly as 
\begin{align*}
	[f]_{\alpha; \Sigma \times I} = \sup_{\substack{(p, t_1) \ne (q,t_2) \in \Sigma \times I \\ q \in B_\delta^\Sigma(p)}} \frac{\abs{f(p) - f(q)}}{d_\Sigma(p,q)^{\alpha} + \abs{t_1 - t_2}^{\alpha/2}}.
\end{align*}
The corresponding parabolic H\"{o}lder space $C^{k,\alpha}_P(\Sigma \times I)$ consists of functions $f$ such that 
\begin{align*}
	\norm{f}_{C^{k,\alpha}_P(\Sigma \times I)} = \sum_{i+2j \le k} \sup_{\Sigma \times I} \abs{\nabla_\Sigma^i \nabla_t^j f} + \sum_{i+2j = k} [\nabla_\Sigma^i \nabla_t^j f]_{\alpha; \Sigma \times I}.
\end{align*}
We shall write $C^{k,\alpha} = C^{k,\alpha}(\Sigma)$ and $C^{k,\alpha}_P = C^{k,\alpha}_P(\Sigma \times (-\infty,0])$ when $\Sigma$ is clear from context.

Finally we define the weighted Sobolev space. For an integer $k$ we say $f \in W^k_{\frac{1}{4}}(\Sigma)$ if
\begin{align*}
	\norm{f}_{W^k_{\frac{1}{4}}} = \norm{f}_{W^k_{\frac{1}{4}}(\Sigma)} = \left(\int_\Sigma \sum_{i=0}^k \abs{\nabla_\Sigma^i f}^2 e^{\frac{\abs{\xX}^2}{4}} d\mathcal{H}^n\right)^{1/2} < \infty.
\end{align*}
 Moreover, we say a map $f(x,s) \in W^{k}_{\frac{1}{4}}(\Sigma \times (-\infty,0])$ if $f(\cdot, t) \in W^k_{\frac{1}{4}}(\Sigma)$ and $f(x,\cdot) \in L^2((-\infty,0))$. In this article we are only interested in the weight $\wt$, so we will write $W^k = W^{k}_{\frac{1}{4}}$ and further write $W^0 = W$. The square bracket $\inner{\cdot, \cdot}$ will denote the standard inner product in $W$, i.e.
 \begin{align*}
 	\inner{f,g} = \int_\Sigma fg \wt,
 \end{align*}
 whereas $\cdot$ will denote the standard inner product in $\mathbb{R}^{n+1}$.
\subsection{Self-expanders}
We recall some basic facts about self-expanders from \cite{BWSpace} and \cite{BWIntegerDegree}. It is useful to use the following variational characterization of self-expanders: A hypersurface $\Sigma \subset \R^{n+1}$ is a self-expander if it is a (formal) critical point of the functional 
\begin{align}
\label{expander-functional}
	E[\Sigma] = \int_{\Sigma} e^{\frac{\abs{\xX}^2}{4}} d\mathcal{H}^n.
\end{align}
The Euler-Lagrange equation of the functional \cref{expander-functional} is precisely \cref{expander-equation}. A simple consequence of equation \cref{expander-equation} is that the second fundamental form $A_\Sigma$ satisfies $\abs{A_\Sigma} = O(\abs{\xX(p)}^{-1})$. \par 
Given a self-expander $\Sigma$ and a compactly supported variation $\Phi_t(\Sigma)$ such that $\left.\frac{d}{dt}\right|_{t=0} \Phi_t(\Sigma) = X$, the second variation formula of the expander functional \cref{expander-functional} is (see eg. \cite[Proposition 4.2]{BWSpace}):
\begin{align*}
	\left.\frac{d^2}{dt^2}\right|_{t=0} E[\Phi_s(\Sigma)] = -\int_{\Sigma} (X \cdot \nN_\Sigma)L_\Sigma(X \cdot \nN_\Sigma) e^{\frac{\abs{\xX}^2}{4}} d\mathcal{H}^n. 
\end{align*}
Here $L_\Sigma$ is the \emph{stability operator} of $\Sigma$ given by 
\begin{align*}
	L_\Sigma = \Delta_\Sigma + \frac{1}{2}\xX \cdot \nabla_\Sigma + \abs{A_\Sigma}^2 - \frac{1}{2}.
\end{align*}
For a connected self-expander $\Sigma$, the \emph{Morse index} of $\Sigma$, $\mathrm{ind}(\Sigma)$, is the biggest dimension of a linear subspace $V \subset C_0^2(\Sigma)$ such that 
\begin{align*}
	-\int_\Sigma  v L_\Sigma v e^{\frac{\abs{x}^2}{4}} d\mathcal{H}^n < 0 \text{ for all } v \in V \setminus \{0\}.
\end{align*}
A number $\mu \in \R$ is an eigenvalue of $-L_\Sigma$ if there exists $f \in W^{2}(\Sigma)$ such that $-L_\Sigma f = \mu f$. By works of Bernstein--Wang \cite{BWIntegerDegree}, when $\Sigma$ is smooth, $L_\Sigma$ has a discrete spectrum and we can therefore order the eigenvalues of $-L_\Sigma$. It follows that the index of $\Sigma$ is equal to the number of negative eigenvalues of $-L_\Sigma$. Let $\mathrm{nul}(\Sigma)$ denote the \textit{nullity} of the operator $-L_\Sigma$, i.e. the multiplicity of eigenvalue 0. $f \in W^2$ is called a \textit{Jacobi field} if $f \in \ker L_\Sigma$.\par  
One of the features of self-expanders is that the eigenfunctions have very fast exponential decay to offset the large exponential weight (this is in contrast with the self-shrinker case where one has to be very careful about the decay rate). Indeed, by \cite[Proposition 2.1]{BWIntegerDegree}, we have 
\begin{align}
\label{eigenfunction-decay}
	\norm{e^{\frac{\beta}{2} \abs{\xX}^2} f}_{C^{k,\alpha}(\Sigma \setminus B_R(0))} < \infty
\end{align}
for any $\beta < \frac{1}{2}$ and $R$ large depending on $\Sigma$. Moreover, the first eigenfunction $\phi_1$ corresponding to the lowest eigenvalue of $-L_\Sigma$ has the fastest decay rate of $(1+\abs{\xX})^{-\frac{1}{2}(n-1-2\lambda_1)} \wt$ by \cite[Proposition 3.2]{BWMountainPass}. \par 

In our study, we will often be working with normal graphs over a given self-expander $\Sigma$. Suppose $v \in C^2(\Sigma)$ with $\norm{v}_{C^2} < \eps$ is such that $\hh = \xX|_\Sigma + v\nN_\Sigma$ is an embedding, then it follows from a computation of Bernstein and Wang \cite[Proposition A.1]{BWRelativeEntropy} that
\begin{align}
	\label{expander-mean-curvature-formula}
	H_{\Sigma^\hh} + \frac{\xX}{2} \cdot \nN_{\Sigma^{\hh}} = -L_{\Sigma} v + Q(v,\xX \cdot \nabla_\Sigma v, \nabla_\Sigma v, \nabla^2_\Sigma v),
\end{align}
where $\Sigma^{\hh} = \hh(\Sigma)$, $H$ is the scalar mean curvature and $Q$ is a homogeneous quadratic polynomial with bounded coefficients. $Q$ can be written as 
\begin{align}
\label{asymptotic-expansion-q}
	Q(s,\rho,d, T) = a(s,\rho,d,T) \cdot d + b(s, d, T)s,
\end{align}
where $a, b$ are homogeneous polynomial of degree 1 with bounded coefficients. Moreover, $a$ and $b$ satisfy the estimates
\begin{equation}
	\label{nonlinear-error-estimate}
\begin{aligned}
	\abs{\xX}^{2+j-\ell}\abs{\nabla_{\Sigma}^i \nabla_s^j \nabla_d^k \nabla_T^{\ell} b(s,d,T)} \le C(\abs{\xX}^{-1}\abs{s} + \abs{d} + \abs{\xX}\abs{T}), \\
	\abs{\xX}^{1+j-\ell}\abs{\nabla_{\Sigma}^i \nabla_s^j \nabla_d^k \nabla_T^{\ell} a(s,\rho,d,T)} \le C(\abs{\xX}^{-1}\abs{s} + \abs{d} + \abs{\xX}\abs{T}).
\end{aligned}
\end{equation}
This follows from the estimates that, on a self-expanding end, $\abs{A_{\Sigma}} = O(\abs{\xX}^{-1})$ and $\abs{\nabla_\Sigma A_{\Sigma}} = O(\abs{\xX}^{-2})$. See \cite[Lemma 3.6]{CCMSGeneric} for a proof in the self-shrinker case, which carries almost identically to our case (see also \cite[Theorem 2.1]{DeruelleSchulze}).

\section{Existence of tame ancient flows}
\label{existence-section}
In this section we establish the existence of tame ancient RMCFs coming out of an unstable, asymptotically conical self-expander. Let $\mathcal{C} \subset \mathbb{R}^{n+1}$ be a smooth cone and fix a smooth self-expander $\Sigma$ asymptotic to $\mathcal{C}$. Let $\tilde{\mathcal{M}} = \{\tilde{\mu}_s\}_{s \in (-\infty,0)}$ be a rescaled Brakke flow such that 
\begin{align}
\label{backward-convergence}
	\lim_{s \to -\infty}  \tilde{\Sigma}_{s} =  \Sigma \text{ in } C^{\infty}_{\loc}(\R^{n+1})
\end{align}
We first establish some general properties satisfied by such flows.
\begin{prop}
\label{properties-of-the-flow}
	Let $\mathcal{M}$ be as above. There is $R = R(\Sigma) > 0$ such that $\tilde{\mathcal{M}}_s \setminus B_R(0)$ is a smooth RMCF. Moreover, $\tilde{\mathcal{M}}_s \setminus B_{2R}(0)$ can be written as a normal graph of $v$ over an open set $V \subset \Sigma \setminus B_R(0)$ such that
\begin{align*}
	\abs{v(p,s)} < C\abs{\xX(p)}^{-1} \text{ for } p \in \Sigma \setminus B_R(0)
\end{align*}
\end{prop} 
\begin{proof}
	The distance estimate follows from using small balls as barriers. See \cite[Proposition A.1]{ChenII} or \cite[Lemma 4.3]{BWTopology}. This can then be upgraded to smoothness using standard pseudolocality arguments. See \cite[Proposition A.2]{ChenII} or \cite[Proposition 4.4]{BWTopology}.
\end{proof}
The key fact we need to use in the subsequent analysis is that all such RMCF, up to a time translation, can be written as an entire normal graph over $\Sigma$. 
\begin{prop}
\label{backwards-convergence}
	Given $\delta > 0$, there is $s_\delta$ such that $\tilde{\mathcal{M}}$ can be written as a normal graph of a smooth function $v$ over $\Sigma$ on $(-\infty,s_\eps)$ such that 
	\begin{align*}
		\norm{v}_{C^{2,\alpha}_P} < \eps \text{ for all } s < s_\eps.
	\end{align*}
\end{prop}
\begin{proof}
	Since the varifold convergence \cref{backward-convergence} has multiplicity 1, we may apply the Brakke regularity theorem (\cite{Brakke} or \cite{WhiteRegularity}) to conclude that the convergence is in fact smooth on compact sets. Moreover, \cref{properties-of-the-flow} shows that there exists $R = R(\Sigma)$ such that $\supp \tilde{\mu}_s \setminus B_{2R}(0)$ can be written as a normal graph $v$ over $\Sigma \setminus B_R(0)$ such that 
	\begin{align*}
		\norm{v(p,s)}_{C^{2,\alpha}(\Sigma \setminus B_R(0))} \le C\abs{\xX(p)}^{-1}.
	\end{align*}
	Now given $\delta$ we choose $R$ sufficiently large so that $CR^{-1} < \eps$ and then $s_\eps$ sufficiently negative such that $\supp \tilde{\mu}_s \cap B_R(0)$ is a normal graph over $\Sigma \cap B_R(0)$ with $C^{2,\alpha}$ norm bounded by $\eps$ as well. A similar argument applied to $\frac{\partial v}{\partial s}$ using the fact that $v$ satisfies the graphical RMCF equation gives the full parabolic estimate (up to decreasing $s_\eps$). 
\end{proof}
\begin{rem}
	Arguing as in \cite[Lemma 7.18]{CCMSGeneric}, it is possible to obtain a weighted estimate as well, but in the self-expander case, we do not need to work with weighted H\"{o}lder space in the following due to the rapid decay of eigenfunctions.
\end{rem}
Now suppose further that $\Sigma$ is unstable. Since $\Sigma$ has a discrete spectrum, $\Sigma$ has finite index which we will denote by $I$. By considering the Rayleigh quotient 
\begin{align*}
	\frac{\int_\Sigma \abs{\nabla_\Sigma v}^2 \wt }{\int_\Sigma \abs{v}^2 \wt }
\end{align*}
for $v \in C^{2,\alpha} \cap W^2(\Sigma) \setminus \{0\}$, it is easy to see that the first eigenvalue is indeed achieved. \par 
We establish, using PDE method similar to \cite{ChoiMantoulidis}, the existence of an $I$-parameter family of ancient solutions to the RMCF starting from $\Sigma$. Each one of these solutions will correspond to a MCF coming out of $\mathcal{C}$ that is not self-similar.  \par 
We work with the RMCF equation \cref{rescaled-mcf-equation}. The tame ancient flows will be constructed as limits of suitable perturbations of the initial hypersurface $\Sigma$ by its eigenfunctions. To this end, we fix an orthonormal basis $\{\phi_i\}_{i=1}^\infty \subset W$ consisting of eigenfunctions of $-L_\Sigma$, where $\phi_{1}, \ldots, \phi_I$ correspond to the negative eigenvalues, $\phi_{I+1}, \ldots, \phi_{I+K}$ correspond to the 0 eigenvalue and $\phi_{I+K+1},\ldots$ correspond to positive eigenvalues. Define the map $\tau_-: \R^I \to W(\Sigma \times [-\infty,0))$ by
\begin{align*}
\tau_-(a_1,\ldots,a_I) = \sum_{i=1}^I a_i e^{-\lambda_i s} \phi_i(x).
\end{align*}
Let $\Pi_-: W \to W$ be the projection onto the negative eigenspace, i.e.
\begin{align*}
	\Pi_-(v) = \sum_{i=1}^I \inner{v,\phi_i} \phi_i(x).
\end{align*}
\par 
Using \cref{expander-mean-curvature-formula}, \cref{rescaled-mcf-equation} reduces to the following nonlinear problem:
\begin{align}
	\label{nonlinear-problem}
	\frac{\partial}{\partial s} v = L_{\Sigma} v + Q(v,\xX \cdot \nabla_\Sigma v, \nabla_\Sigma v, \nabla^2_\Sigma v).
\end{align}
To solve this nonlinear PDE, we first consider the inhomogeneous linear PDE:
\begin{align}
	\label{linear-problem}
	\frac{\partial}{\partial s} v = L_\Sigma v + h(p,s), \; \; (p,s) \in \Sigma \times (-\infty,0),
\end{align}
where $h(p,s)$ is a smooth function. We will assume in addition that $h(p,s)$ is exponentially decaying in time, i.e. there exists some $\delta > 0$ such that
\begin{align}
	\label{inhomogeneous-decay}
		\int_{-\infty}^0 \abs{e^{-\delta s} \norm{h(\cdot,s)}_{W}}^2 ds  < \infty. 
\end{align}
Using separation of variables, we can write a solution to \cref{linear-problem} as
\begin{align*}
	v(p,s) = \sum_{i=1}^\infty v_i(p) \phi_i(s),
\end{align*}
where $v_i$ formally solves the ODE
\begin{align*}
	v_i'(s) = -\lambda_iv_i(s) + \sum_{j=1}^\infty h_j(s) \phi_j(p),
\end{align*}
where $h_j(s) = \inner{h(\cdot,s),\phi_j}$. This allows us to establish the existence of the $L^2$-solutions to \cref{linear-problem}.
\begin{prop}[cf. Lemma 3.1 of \cite{ChoiMantoulidis}]
\label{existence-linear-problem}
	Let $(a_1,\ldots,a_I) \in \mathbb{R}^I$ and suppose $h$ satisfies \cref{inhomogeneous-decay}. There exists a unique solution $v$ to \cref{linear-problem} such that $\Pi_-(v(\cdot,0)) = \sum_{i=1}^I a_i \phi_i$ and 
	\begin{align}
	\label{linear-problem-decay}
		\int_{-\infty}^0 \abs{e^{-\delta' s} \norm{v(\cdot,s)}_{W^{1}}}^2 ds  < \infty,
	\end{align}
	for some $0 < \delta' < \min\{\delta, -\mu_I\}$. $v$ can be written as 
	\begin{align*}
		v(p,s) = \sum_{i=1}^\infty v_i(s) \phi_i(p),
	\end{align*}
	where
	\begin{align*}
		v_i(s) = \begin{cases}
			a_ie^{-\lambda_is} - \int_s^0 e^{\lambda_i(\sigma - s)} h_i(\sigma)d\sigma & i = 1, \ldots, I,  \\
			\int_{-\infty}^s e^{\lambda_i(\sigma-s)} h_i(\sigma)d\sigma & i = I+1,I+2,\ldots
		\end{cases}.
	\end{align*}
	Here $h_i(s) = \inner{h(p,s),\phi_i(x)}$. Additionally, $v$ satisfies the estimate 
	\begin{align*}
		e^{-\delta' s}\norm{v(\cdot,s) - \tau_-(a_1,\ldots,a_I)}_{W} \le C\left(\int_{-\infty}^0 \abs{e^{-\delta\sigma}\norm{h(\cdot,\sigma)}_{W}}^{2}d\sigma\right)^{\frac{1}{2}}
	\end{align*}
	for all $s < 0$.
\end{prop}
\begin{proof}
	Let us check that $v$ solves \cref{linear-problem}. Indeed, differentiating under the integral sign gives
	\begin{align*}
		\frac{\partial}{\partial s} v(s) &= \sum_{i=1}^I \left[-\lambda_i a_i e^{-\lambda_i s} + h_i(s) + \lambda_i \int_s^0 e^{\lambda_i(\sigma - s)} h_i(\sigma) d\sigma\right] \phi_i \\
		&\phantom{{}={}} + \sum_{i=I+1}^\infty \left[ h_i(s) - \lambda_i \int_{-\infty}^s e^{\lambda_i(\sigma - s)}h_i(\sigma) d\sigma\right]\phi_i
	\end{align*}
	Using $h(p,s) = \sum_{i=1}^\infty h_i(s)\phi_i(p)$ and $\lambda_i \phi_i = -L_\Sigma \phi_i$ we get that 
	\begin{align*}
		\frac{\partial}{\partial s} v(s) = &\phantom{{}+{}}\sum_{i=1}^I L_\Sigma \phi_i \left(a_i e^{-\lambda_i s} - \int_s^0 e^{\lambda_i(\sigma -s )}h_i(\sigma) d\sigma \right) \\
		&+ \sum_{i=I+1}^\infty  L_\Sigma \phi_i \int_{-\infty}^s e^{\lambda_i(\sigma -s )}h_i(\sigma) d\sigma + h(\cdot,s).
	\end{align*}
	Evidently this is \cref{linear-problem}. The fact that $\Pi_-(v(\cdot,0)) = \sum_{i=1}^I a_i\phi_i$ is a straightforward computation. The uniqueness of the solution follows from the uniqueness of the linear parabolic equation $\frac{\partial}{\partial s} v = L_\Sigma v$ (with 0 initial data). \cref{linear-problem-decay} follows from the assumption on $h$ and the fact that, if $\lambda_i < 0$,
	\begin{align*}
		\int_{-\infty}^0 e^{-2\delta' s - 2\lambda_i s} ds < \infty 
	\end{align*}
	as long as $\delta' < -\lambda_i$. Finally the last estimates follow from H\"{o}lder's inequality as the subtraction kills the $a_ie^{-\lambda_i s}$ terms.
\end{proof}

We now follow the ideas of \cite{BWSpace} to establish higher regularity of the solutions obtained above. First notice that for a given initial data $a = (a_1,\ldots,a_I)$, $\tau_{-}(a_1,\ldots,a_I)$ solves the linear homogeneous equation $\frac{\partial}{\partial s} v = L_\Sigma v$, hence by replacing $v$ by $v- \tau_{-}(a_1,\ldots,a_I)$ we will WLOG assume that $v$ is a solution to \cref{linear-problem} with 
\begin{align}
	\label{zero-initial-data}
	\Pi_{-}(v(\cdot,0)) = 0.
\end{align}	
We will prove a H\"{o}lder estimate of the following type:
\begin{prop}
\label{schauder-estimate}
	Suppose that $h \in C^{0,\alpha}_P(\Sigma \times (-\infty,0])$. Let $v$ be the solution to \cref{linear-problem} with \cref{zero-initial-data} constructed in \cref{existence-linear-problem}, then for $s < 0$,
	\begin{align*}
		[\nabla_\Sigma^2 v]_{\alpha; \Sigma \times [s-1,s]} + \sup_{\Sigma \times [s-1,s]} \left(\abs{\nabla_\Sigma v} + \abs{\nabla_\Sigma^2 v}\right)  \le C \norm{h}_{C^{0,\alpha}_P(\Sigma \times (-\infty,0])}.
	\end{align*}
\end{prop}
\begin{rem}
	We emphasize that, in the self-expander case, all eigenfunctions are decreasing rapidly in the spatial variables by \cref{eigenfunction-decay}, and therefore we do not need to invoke any weighted Schauder estimates. This is in contrast to the self-shrinker case in \cite{CCMSGeneric} as the eigenfunctions of the stability operator of self-shrinkers can potentially have polynomial growth in space, for which weighted estimates are necessary.
\end{rem}
\begin{proof}
	WLOG we assume $s = 0$. In this proof it is often convenient to work with the original MCF. To this end let $\Sigma_t = \sqrt{t}\Sigma$ denote the MCF associated to the self-expander $\Sigma$.  Since $\Sigma$ is asymptotically conical, we can find $R > 0$ such that, for any $p \in \Sigma \setminus B_R(0)$ and $t \in [e^{-1},1]$, $\Sigma_t \cap B_2(p)$ can be parametrized as a normal graph over some open subset of $T_p\Sigma$ with small $C^3$ norm, i.e. there exists $\psi_p: \Omega \times [e^{-1},1] \subset T_p\Sigma \times [0,1] \to \R$ such that $\psi_p(0,1) = 0$,
	\begin{align}
	\label{transfer-estimate}
		\sum_{i=0}^3 \abs{\nabla^i_{\R^n} \psi_p} + \sum_{i=0}^1 \abs{\partial_t \nabla^i_{\R^n}\psi_p} \le \eps,
	\end{align}
	and 
	\begin{align*}
		\Psi_p(x,t) = \xX(p) + \xX(x) + \psi_p(x,t) \nN_\Sigma(p)
	\end{align*}
	parametrizes $\Sigma_t \cap B_2(p)$, $t \in [e^{-1},1]$. \par 
	We now fix $p \in \Sigma \setminus B_R(0)$ and consider, on $\Omega \times [e^{-1},1]$,
	\begin{align*}
		w(x,t) = t^{\frac{1}{2}} v(t^{-\frac{1}{2}}\Psi_p(x,t), \log t).
	\end{align*}
	By \cref{linear-problem} and the chain rule we see that $w$ satisfies the equation 
	\begin{align*}
		\frac{\partial w}{\partial t} - \Delta_{\Sigma_t} w - \abs{A_{\Sigma_t}}^2 w =  t^{-\frac{1}{2}} h(t^{-\frac{1}{2}} \Psi_p(x,t), \log t),
	\end{align*}
	 Since $h \in C_P^{0,\alpha}$, $\abs{\Psi_p(x,t)} < 2\abs{\xX(p)}$, and $t \in [e^{-1},1]$ which is compact, it follows that the right hand side $g_p = t^{-\frac{1}{2}} h(t^{-\frac{1}{2}} \Psi_p(x,t), \log t)$ is also H\"{o}lder continuous in spacetime with 
	\begin{align*}
		\abs{g_p(x,t_1) - g_p(y,t_2)} \le C \norm{h}_{C^{0,\alpha}_P}(\abs{x-y}^\alpha + \abs{t_1 - t_2}^{\frac{\alpha}{2}}).
	\end{align*}
	We may now invoke interior Schauder estimates (see eg. \cite{LSU}) to conclude that (note that $\abs{A_{\Sigma_t}}^2 \le C\abs{\xX}^{-2}$ on $[e^{-1},1]$)
	\begin{align*}
		\sup_{\Omega \times [e^{-1},1]} \left(\abs{\nabla_{\R^n} w} + \abs{\nabla^2_{\R^n} w}\right) + [\nabla_{\R^n}^2 w]_{\alpha; \Omega \times[e^{-1},1]}\le C \norm{h}_{C^{0,\alpha}_P(\Sigma \times (-\infty,0])}.
	\end{align*}
	Rescaling back and using \cref{transfer-estimate}, we see that 
	\begin{align*}
		\sup_{s \in [-1,0]}\left( \abs{\nabla_\Sigma v(p,s)} + \abs{\nabla_\Sigma^2 v(p,s)}\right) + [\nabla_\Sigma^2 v]_{\alpha;\Omega \times [-1,0]}\le C \norm{h}_{C^{0,\alpha}_P}.
	\end{align*}
	As $p \in \Sigma \setminus B_R(0)$ is arbitrary, we see the desired estimates hold on $(\Sigma \setminus B_R(0)) \times [-1,0]$. The estimates on $\Sigma \cap B_R(0)$ follows from standard parabolic Schauder estimates for compact domains. 
\end{proof}
It is standard to bootstrap \cref{schauder-estimate} to obtain higher order Schauder estimates, and a simple modification of the proof gives the corresponding estimates for $\partial_s v$. These lead to the following exponentially weighted variant, which follows easily from \cref{schauder-estimate}.
\begin{cor}
\label{main-schauder-estimates}
	Suppose that there is $\delta > 0$ such that
	\begin{align*}
		\sup_{s < 0} e^{-\delta s} \norm{h}_{C^{0,\alpha}_P(\Sigma \times [s - 1, s])} < \infty.
	\end{align*}
	Let $v$ be the solution to \cref{linear-problem} with \cref{zero-initial-data} constructed in \cref{existence-linear-problem}. Then for every $0< \delta' < \min\{\delta, -\lambda_I\}$ we have, for $s < 0$
	\begin{align*}
		e^{-\delta' s} \norm{v}_{C^{2,\alpha}_P(\Sigma \times [s-1,s])} \le C \sup_{\sigma < 0} e^{-\delta \sigma} \norm{h}_{C^{0,\alpha}_P(\Sigma \times [\sigma - 1, \sigma])}.
	\end{align*} 
\end{cor}
We are now in position to prove the existence and uniqueness theorem for \cref{nonlinear-problem}. \cref{main-theorem} follows easily from the following. 
\begin{thm}[cf. Theorem 3.3 of \cite{ChoiMantoulidis}]
	\label{ancient-solution-existence}
	Let $\delta_0 \in (0,-\lambda_I)$. For every sufficiently large $\beta$, there exists $\eps = \eps(\Sigma,\delta_0,\beta)$ such that the following holds: for any $(a_1,\ldots,a_I) \in B_\eps^I(0)$, there exists a unique ancient solution $v$ to \cref{nonlinear-problem} satisfying $\Pi_{-}v(\cdot, 0) = \sum_{i=1}^I a_i \phi_i$ and
	\begin{align}
		\label{closeness-estimate}
		e^{-\delta_0 s}\norm{v - \tau_{-}(a_1,\ldots,a_I)}_{C^{2,\alpha}_{P}(\Sigma \times [s-1,s])} \le \beta \sum_{i=1}^I a_i^2
	\end{align}
	for all $s \le 0$.
\end{thm}
\begin{proof}
	In this proof we will also treat the nonlinear term $Q$ as a function of $s$ and $p \in \Sigma$, and will write $Q(p,s)$ whenever we do so. Let us fix $a = (a_1,\ldots,a_I) \in B_\eps^I(0)$. Consider the space 
	\begin{align*}
		C^* = &\{v: \Sigma \times (-\infty,0] \to \R \mid  \\
		&e^{-\delta_0 s} \norm{v}_{C^{2,\alpha}_{P}(\Sigma \times [s-1,s])} < \infty \text{ for all } s \le 0 \text{, and } \Pi_{-}(v(\cdot,0)) = \sum_{i=1}^I a_i\phi_i\}.
	\end{align*}
	Evidently $C^*$ is a Banach space equipped with the norm
	\begin{align*}
		\norm{v}_{C^*} =  \sup_{s < 0} e^{-\delta_0 s} \norm{v}_{C^{2,\alpha}_{P}(\Sigma \times [s-1,s])}.
	\end{align*} 
	Given $v \in C^* \cap C^\infty(\Sigma \times (-\infty,0])$ we let $\Psi(v, a)$ be a solution to the linear problem 
	\begin{align*}
		\left(\frac{\partial}{\partial t} - L_\Sigma\right)\Psi(v,a) = Q(v, \xX \cdot \nabla_\Sigma v, \nabla_\Sigma v, \nabla_\Sigma^2 v). 
	\end{align*}
	Using the estimates on the nonlinear term $Q$, \cref{nonlinear-error-estimate}, we can find $\eta$ sufficiently small depending on $\Sigma$ and $\alpha$  such that
	\begin{align*}
		\norm{v}_{C^*} < \eta \implies  \int_{-\infty}^0 \abs{e^{-\delta_0 s} \norm{Q(\cdot,s)}_{W}}^2 ds  < \infty. 
	\end{align*}
	This allows us to apply \cref{linear-problem-decay} to conclude that $\Psi(v,a)$ is well-defined (i.e. the solution exists and is unique). The function 
	\begin{align*}
		\tilde{\Psi}(v,a) = \Psi(v,a) - \sum_{i=1}^I a_i e^{-\lambda_i s} \phi_i
	\end{align*}
	then (uniquely) solves the problem 
	\begin{align*}
		\begin{cases}
			\left(\frac{\partial}{\partial s} - L_\Sigma\right)w &=  Q(v, \xX \cdot \nabla_\Sigma v, \nabla_\Sigma v, \nabla_\Sigma^2 v) \text{ on } \Sigma \times (-\infty,0]  \\
			\Pi_{-} w(\cdot, 0) &= 0
		\end{cases}.
	\end{align*}
	Using the parabolic Schauder estimates \cref{main-schauder-estimates}, we get that for all $s < 0$
	\begin{align*}
		&\phantom{{}\le{}}e^{-\delta_0 s}\norm{\tilde{\Psi}(v,a)}_{C^{2,\alpha}_P(\Sigma \times [s-1,s])} \\
		&\le C\sup_{\sigma < 0} e^{-\delta_0 \sigma}\norm{Q(v,\nabla_\Sigma v, \xX \cdot \nabla_\Sigma v, \nabla^2_\Sigma v)}_{C^{0,\alpha}_P(\Sigma \times [\sigma-1,\sigma])}. 
	\end{align*}
	Using \cref{nonlinear-error-estimate} again, we see that, upon taking supremum on the left hand side 
	\begin{align*}
		\norm{\tilde{\Psi}(v,a)}_{C^*} \le C \sup_{\sigma \le 0} e^{-\delta_0 \sigma} \norm{v}_{C^{2,\alpha}_P(\Sigma \times [\sigma-1, \sigma])}.
	\end{align*}
	This shows that $\bar{\Psi}(v,a)$ is a well-defined map from $C^*$ to itself provided $\norm{v}_{C^*}$ is sufficiently small. \par 
	Repeating the above argument for another function $w$ with $\norm{w}_{C^*}$ sufficiently small shows that 
	\begin{align*}
		\norm{\Psi(v,a) - \Psi(w,a)}_{C^*} \le C(\norm{v}_{C^*} + \norm{w}_{C^*}) \norm{v -w }_{C^*}.
	\end{align*}
	Thus $\Psi(v,a)$ is a continuous contraction mapping on $C^*$ provided $\eta$ is chosen small enough. Hence by the contraction mapping theorem, there exists a unique fixed point $\Psi(a)$ of the map $\Psi(\cdot, a)$ in $C^*$. This $\Psi(a)$ is the solution we seek for the nonlinear problem \cref{nonlinear-problem} with $\Pi_- \Psi(a) = \sum_{i=1}^I a_i \phi_i$. The regularity of $\Psi(a)$ follows from standard parabolic regularity theory (with \cref{main-schauder-estimates}).
\end{proof}

\section{Relative expander entropy and uniqueness}
\label{uniqueness-section}
In this section we define a version of the relative expander entropy adapted to a normal graph, which will be suitable in our setting, and prove several useful properties. We will use the relative expander entropy to deduce the uniqueness theorem, \cref{main-uniqueness-theorem}. Throughout the section, fix a smooth cone $\mathcal{C} \subset \R^{n+1}$ and a self-expander $\Sigma$ asymptotic to $\mathcal{C}$. \par
Let $\chi_R: \R^{n+1} \to \R$ be smooth cutoff functions supported on $B_{R+2}(0)$ and identically 1 on $B_{R}(0)$. For any function $v \in C^{2,\alpha}_0(\Sigma)$, define the \textit{relative expander entropy} of $\Sigma_v$ to be 
\begin{align*}
	E_{\rel}^*[\Sigma_v,\Sigma] = \int_{\Sigma_v} \wt d\mathcal{H}^n - \int_{\Sigma} \wt d\mathcal{H}^n.
\end{align*}
For a general function $v \in W^1 \cap C^{2,\alpha}(\Sigma)$, define $E_{\rel}^*[\Sigma_v,\Sigma]$ by
\begin{align*}
	E_{\rel}^*[\Sigma_v,\Sigma] = \lim_{R \to \infty} E_{\rel}^*[\Sigma_{\chi_R v},\Sigma]
\end{align*}
whenever the limit exists. \par 
We note importantly that this notion of relative expander entropy is slightly different from the usual definition of the relative entropy from \cite{BWRelativeEntropy} and \cite{DeruelleSchulze}. We will denote by $E_{\rel}$ the quantity defined by Bernstein--Wang; $E_{\rel}$ is given by the formula
\begin{align*}
	E_{\rel}[\Sigma_1,\Sigma_2] = \lim_{R \to \infty} \left(\int_{\Sigma_1 \cap B_R(0)} \wt d \mathcal{H}^n - \int_{\Sigma_2 \cap B_R(0)} \wt d \mathcal{H}^n\right),
\end{align*}
for two hypersurfaces $\Sigma_1$ and $\Sigma_2$, whenever the limit is defined (possibly $\infty$). In particular, they showed in \cite{BWRelativeEntropy} that when $\Sigma_1$ is a hypersurface trapped between two self-expanders asymptotic to the same cone $\mathcal{C}$, then $E_{\rel}[\Sigma_1, \Gamma]$ is well-defined (possibly $\infty$, but not $-\infty$) for any self-expander $\Gamma$ asymptotic to $\mathcal{C}$. Because of this, $E_{\rel}$ is the natural and more suitable quantity to study in the trapped case (and, in fact, $E_{\rel} = E_{\rel}^*$ in the trapped case - see \cref{entropy-equivalence}). Unfortunately, in order for a graph $\Sigma_v$ to be trapped, $v$ needs to have a very good spatial decay near infinity:
\begin{align*}
	v(p) = O(\abs{\xX(p)}^{-n-1} e^{-\frac{\abs{\xX(p)}^2}{4}}).
\end{align*}
As we are working with normal graphs that a priori do not have the sharp decay (rather only an energy bound), we will have to use $E_{\rel}^*$ instead of $E_{\rel}$ for now. See \cref{cone-section} for further discussion on the trapping assumption and the difference between $E_{\rel}$ and $E^*_{\rel}$. \par 
We will show that $E_{\rel}^*$ is well-defined if the function $v$ has small $C^{2,\alpha}$ norm.
\begin{prop}
	\label{motivating-computation}
	Suppose $v \in C^{2,\alpha}_0(\Sigma)$ function. There exists $\eps = \eps(\Sigma)$ sufficiently small such that the following inequality holds:
	\begin{align}
		\label{expansion-formula}
		\abs{E_{\rel}^*[\Sigma_v,\Sigma] - \frac{1}{2} \int_\Sigma (\abs{\nabla_\Sigma v}^2 + (\frac{1}{2} - \abs{A_\Sigma}^2) v^2) \wt} \le C \eps \norm{v}_{W^1}^2
	\end{align}
	whenever $\norm{v}_{C^{2,\alpha}} < \eps$. Here $C = C(\Sigma)$. 
\end{prop}
\begin{rem}
	Integrating by parts (which is justified as $v$ is compactly supported) gives that the second order term in \cref{expansion-formula} is exactly equal to 
	\begin{align}
		\label{second-order-term}
		- \frac{1}{2}\int_\Sigma vL_\Sigma v  \wt.
	\end{align} 
	This is not surprising as $\Sigma_v$ can be thought of as a perturbation of $v$ when $\Sigma_v$ is sufficiently close to $\Sigma$. Since $\Sigma$ is $E$-stationary, the relative entropy should pick up the second-order information, which is precisely the stability operator.
\end{rem}
\begin{proof}
	We will proceed by explicit computation. By the area formula we can write
	\begin{align}
		\label{relative-area-expansion}
		\int_{\Sigma_v} \wt d\mathcal{H}^n = \int_{\Sigma} \sqrt{\det((Dv)^T(Dv))} e^{\frac{\abs{\xX}^2}{4}} e^{\frac{v^2}{4}} e^{-vH_\Sigma},
	\end{align}
	where we used the self-expander equation \cref{expander-equation}. Here  
	\begin{align*}
		Dv = I_{(n+1) \times n} + \nabla_\Sigma v \otimes \nN_\Sigma + v \nabla_\Sigma \nN_\Sigma,
	\end{align*}
	where $I_{(n+1) \times n}$ denotes the $(n+1)$-by-$n$ matrix which is the identity in the top $n$ rows and 0 in the last row. Given $p \in \Sigma$, the matrix $Dv$ can be written in normal coordinates centered at $p$ as 
	\begin{align*}
		Dv = \begin{pmatrix}
			1+ v\kappa_1 & \cdots & 0 \\
			\vdots & \ddots & \vdots \\
			0 & \cdots & 1 + v\kappa_n \\
			\partial_1 v  & \cdots & \partial_n v 
		\end{pmatrix} ,
	\end{align*}
	where $\kappa_1, \ldots, \kappa_n$ are the principle curvatures of $\Sigma$. Hence the Jacobian matrix $(Dv)^T(Dv)$ takes the form
	\begin{align*}
		\begin{pmatrix}
			(1+v \kappa_1)^2 + (\partial_1 v)^2 & \partial_1v \partial_2 v & \cdots & \partial_1 v \partial_n v \\
			\partial_1v\partial_2 v & (1 + v\kappa_2)^2 + (\partial_2 v)^2 & \cdots & \partial_2v \partial_n v \\
			\vdots & \ddots& & \vdots \\
			\partial_1v \partial_n v& \partial_2v \partial_nv & \cdots & (1+v\kappa_n)^2 + (\partial_n v)^2 \\
		\end{pmatrix}
	\end{align*}
	Equivalently, the above can be written as $I_n + A$, where $A$ has entries
	\begin{align*}
		A_{ij} = \partial_i v \partial_j v + (v^2\kappa_i^2 + 2v \kappa_i)\delta_{ij}
	\end{align*}
	We now expand the determinant using its series expansions, whose validity is justified by the fact that $\norm{v}_{C^{2,\alpha}} < \eps$:
	\begin{align*}
		\det(I + \eps A) = 1 + \tr(A) + \frac{1}{2}\eps^2 (\tr^2(A) - \tr(A^2)) + O(\eps^3).
	\end{align*}
	In the following we will use $M(v,\nabla_\Sigma v)$ to denote a polynomial of degree at least 3 in $v$ and $\nabla_\Sigma v$ whose coefficients depend only on $\Sigma$. The exact form of $M$ may change from line to line. We compute that 
	\begin{gather*}
		\tr(A) = \abs{\nabla_\Sigma v}^2 + \abs{A_\Sigma}^2 v^2 + 2vH_\Sigma \\
		\tr(A^2) = \sum_{i,j} \abs{\partial_i v \partial_j v}^2 + \sum_{i=1}^n, \left(2\abs{\partial_i v}^2(v^2 \kappa_i^2 + 2v \kappa_i) + v^4 \kappa_i^4 + 4v^3 \kappa_i^3\right) + 4v^2 \abs{A_\Sigma}^2.
	\end{gather*}
	Thus 
	\begin{align*}
		\tr^2(A) - \tr(A^2) = 4(\abs{H_\Sigma}^2 - \abs{A_\Sigma}^2)v^2 + M(v, \nabla_\Sigma v).
	\end{align*}
	Putting the above computations together gives 
	\begin{align}
		\label{determinant-expansion}
		\det(I+A) = 1 + \abs{\nabla_\Sigma v}^2 + 2H_\Sigma v +  (2\abs{H_\Sigma}^2 - \abs{A_\Sigma}^2) v^2 + M(v, \nabla_\Sigma v).
	\end{align}
	Finally, using the Taylor expansion
	\begin{align*}
		\sqrt{1+x} = 1 + \frac{1}{2}x - \frac{1}{8} x^2 + O(x^3),
	\end{align*}
	we get that 
	\begin{align*}
		\sqrt{\det(I+A)} = 1 + \frac{1}{2} \abs{\nabla_\Sigma v}^2 + H_\Sigma v + \frac{1}{2}(\abs{H_\Sigma}^2 - \abs{A_\Sigma}^2) v^2 + M(v, \nabla_\Sigma v).
	\end{align*}
	Using the above in \cref{relative-area-expansion} together with a Taylor expansion on the exponential terms, we obtain (where we wrote $\eps = \norm{v}_{C^{2,\alpha}}$)
	\begin{align}
		&\phantom{{}={}}\int_{\Sigma_v} \wt d\mathcal{H}^n \nonumber \\
		&= \int_\Sigma (1 + \frac{1}{2} \abs{\nabla_\Sigma v}^2 + H_\Sigma v + \frac{1}{2}(\abs{H_\Sigma}^2 - \abs{A_\Sigma}^2) v^2 + M(v, \nabla_\Sigma v)) \nonumber\\
		&\phantom{{}={}\int_\Sigma}(1 + \frac{1}{4}v^2 - vH_\Sigma + \frac{1}{2}v^2\abs{H_\Sigma}^2 M(v, \nabla_\Sigma v)) \wt \nonumber \\
		&= \int_\Sigma (1 + \frac{1}{2}\abs{\nabla_\Sigma v}^2 + \frac{1}{2}(\frac{1}{2}  - \abs{A_\Sigma}^2) v^2+ M(v, \nabla_\Sigma v))\wt \label{expansion-1}
	\end{align}
	Subtracting the above from $\int_\Sigma \wt$ and using that 
	\begin{align*}
		\int_\Sigma M(v, \nabla_\Sigma v) \wt \le C\norm{v}_{C^{2,\alpha}} \norm{v}_{W^1}^2 \le C\eps \norm{v}_{W^1}^2
	\end{align*}
	gives the desired formula. 
\end{proof}
\begin{cor}
	\label{finiteness}
	There exists $\eps> 0$ such that if $v \in W^1 \cap C^{2,\alpha}(\Sigma)$ satisfies $\norm{v}_{C^{2,\alpha}} < \eps$, then $-\infty < E_{\rel}^*[\Sigma_v,\Sigma] < \infty$. 
\end{cor}
\begin{proof}
	For a general $v$, inserting $\chi_Rv$ in place of $v$ in \cref{expansion-1} yields
	\begin{align}
	\label{expansion-2}
		E_{\rel}^*[\Sigma_{\chi_R v}, \Sigma] &= \frac{1}{2} \int_\Sigma  (\abs{\nabla_\Sigma (\chi_Rv)}^2 + (\frac{1}{2} - \abs{A_\Sigma}^2) v^2 \chi_R^2 + M(\chi_Rv, \nabla_\Sigma \chi_R v))  \wt\nonumber\\
		&= \frac{1}{2} \int_{\Sigma} ((\abs{\nabla_\Sigma \chi_R}^2 + \frac{1}{2} - \abs{A_\Sigma}^2)v^2 + \chi_R^2 \abs{\nabla_\Sigma v}^2 \nonumber\\
		& \phantom{{}\le{}} +2 \chi_R v \nabla_\Sigma \chi_R \cdot \nabla_\Sigma v  + M(\chi_Rv, \nabla_\Sigma \chi_R v)) \wt 
	\end{align}
	Thus for $R_1 > R_2$, we get from \cref{expansion-2}
	\begin{align*}
		&\phantom{{}={}}E_{\rel}^*[\Sigma_{\chi_{R_1} v}, \Sigma] -E_{\rel}^*[\Sigma_{\chi_{R_2}v}, \Sigma] \\
		&= \frac{1}{2} \int_\Sigma ( (\abs{\nabla_\Sigma \chi_{R_1}}^2 - \abs{\nabla_\Sigma \chi_{R_1}}^2 )v^2 + (\chi_{R_1}^2 - \chi_{R_2}^2) \abs{\nabla_\Sigma v}^2 \\
		&\phantom{{}={}}	+ 2\chi_{R_1} v \nabla_\Sigma \chi_{R_1} \cdot \nabla_\Sigma v - 2\chi_{R_2} v \nabla_\Sigma \chi_{R_2} \cdot \nabla_\Sigma v \\
		&\phantom{{}={}} + M(\chi_{R_1} v, \nabla_\Sigma \chi_{R_1} v) - M(\chi_{R_2} v, \nabla_\Sigma \chi_{R_2} v) )\wt.
	\end{align*}
	As $v \in W^2$, given $\eta$, there is $R_0$ such that for $R_1, R_2 > R_0$,
	\begin{align*}
		\int_{B_{R_1}(0) \setminus B_{R_2}(0)} (\abs{\nabla_\Sigma v}^2 + v^2)\wt < \eta.
	\end{align*}
	As $\nabla_\Sigma \chi_{R_1} - \nabla_{\Sigma} \chi_{R_2}$ is supported on $\R^{n+1} \setminus B_{R_2}(0)$ and $\chi_{R_1} - \chi_{R_2}$ is supported on $B_{R_1+2}(0) \setminus B_{R_2+2}(0)$, there is $R_0$ such that $R_1 > R_2 > R_0$ implies 
	\begin{gather*}
		\int_\Sigma \abs{(\abs{\nabla_\Sigma \chi_{R_1}}^2 - \abs{\nabla_\Sigma \chi_{R_1}}^2 )v^2 + (\chi_{R_1}^2 - \chi_{R_2}^2) \abs{\nabla_\Sigma v}^2) }\wt  < \frac{\eta}{3}, \\
		\int_\Sigma \abs{2\chi_{R_1} v \nabla_\Sigma \chi_{R_1} \cdot \nabla_\Sigma v - 2\chi_{R_2} v \nabla_\Sigma \chi_{R_2} \cdot \nabla_\Sigma v}\wt < \frac{\eta}{3}.
	\end{gather*}
	For the error term with $M$, one has to first rewrite it into the remainder form (in particular the expansion for $\sqrt{1+x}$ introduces infinitely many terms, which is undesirable in this proof), and then uses the same argument as above to conclude 
	\begin{align*}
		\int_\Sigma \abs{ M(\chi_{R_1} v, \nabla_\Sigma \chi_{R_1} v) - M(\chi_{R_2} v, \nabla_\Sigma \chi_{R_2} v)} \wt < \frac{\eta}{3}
	\end{align*}
	Hence $E_{\rel}^*[\Sigma_{\chi_R v},\Sigma]$ is a Cauchy sequence, and it follows that
	\begin{align*}
		-\infty < E_{\rel}^*[\Sigma_v, \Sigma] = \lim_{R \to \infty} E_{\rel}^*[\Sigma_{\chi_Rv}, \Sigma] < \infty& \qedhere.
	\end{align*}
\end{proof}
Another immediate consequence of \cref{motivating-computation} is the reverse Poincar\'{e} inequality for $E_{\rel}^*$.
\begin{cor}
	\label{reverse-poincare}
	If $v \in W^1 \cap C^{2,\alpha}(\Sigma)$ satisfies $\norm{v}_{C^{2,\alpha}} < \eps$, then
	\begin{align*}
		E_{\rel}^*[\Sigma_v, \Sigma] \ge C_1 \int_{\Sigma} \abs{\nabla_\Sigma v}^2 \wt - C_2 \int_{\Sigma} v^2 \wt.
	\end{align*} 
	for $C_1, C_2$ depending only on $\Sigma$. Specifically, if $E_{\rel}^*[\Sigma_v,\Sigma] \le 0$,
	\begin{align*}
		\int_\Sigma v^2 \wt \ge C \int_\Sigma \abs{\nabla_\Sigma v}^2 \wt.
	\end{align*}
\end{cor}
\begin{proof}
	This follows from the expansion \cref{expansion-1}, using the fact that $\abs{A_\Sigma}$ is bounded on $\Sigma$ and the bound
	\begin{align*}
		\int_{\Sigma} M(v,\nabla_\Sigma v) \wt \ge -C\eps \norm{v}_{W}^2 - C\eps\norm{\nabla_\Sigma v}_W^2. & \qedhere
	\end{align*}
\end{proof}
\par

Since self-expanders are only formally critical points of \cref{expander-functional}, we will treat the RMCF equation \cref{rescaled-mcf-equation} as a gradient flow of the relative expander entropy when the background self-expander $\Sigma$ is fixed and the flow can be written as an entire graph over $\Sigma$.  To prove \cref{main-uniqueness-theorem}, we will adapt the spectral analysis of Choi--Mantoulidis \cite{ChoiMantoulidis} to the self-expander setting. Since $\Sigma$ is not compact, we will use $E_{\rel}^*$ to give the equation \cref{rescaled-mcf-equation} a gradient flow structure. \par

Let $\Sigma$ be an unstable self-expander asymptotic to some smooth cone $\mathcal{C}$. Let $I = \mathrm{ind}(\Sigma)$ and again fix an orthonormal basis $\{\phi_i\}_{i=1}^\infty \subset W^2(\Sigma)$ as before. Let $v(p,s): \Sigma \times (-\infty,0] \to \R$ be a solution to the nonlinear problem \cref{nonlinear-problem}. \cref{main-uniqueness-theorem} is a consequence of the following spectral uniqueness theorem. Recall that we are using the notation $\Pi_-$ as in \cref{existence-section}. 
\begin{thm}
\label{strong-uniqueness}
	There exists $\eps = \eps(\Sigma) > 0$ such that the following hold: suppose $v$ solves \cref{nonlinear-problem} and further satisfies
	\begin{align}
	\label{smooth-convergence}
		\norm{v}_{C^{1,\alpha}_P} \le 1 \text{, } \norm{v(\cdot,s)}_{C^{2,\alpha}} < \eps \text{ for all } s < 0,
	\end{align}
	and
	\begin{align}
	\label{negative-dominant}
		\norm{v(\cdot,s)}_{W} \le C \norm{\Pi_- v(\cdot ,s)}_{W}
	\end{align}
	then, up to a time translation, $v$ agrees with $\Phi(a)$ for some $a = (a_1,\ldots,a_I)$ from \cref{ancient-solution-existence}.
\end{thm}
Given a solution $v$ to \cref{nonlinear-problem}, we also introduce the following extra notations for projections throughout the section: Given $\mu \in \R$,
\begin{align*}
	\Pi_{\sim \mu} v = \sum_{\lambda_i \sim \mu} \inner{v, \phi_i} \phi_i  \text{ and } V_{\sim \mu}(s) = \norm{\Pi_{\sim \mu}v(\cdot,s)}_{W}. 
\end{align*}
Here $\sim$ refers to any of $>, <$ or $=$. When $\mu = 0$, we will instead use $V_+$, $V_-$ and $V_0$ instead of $V_{> 0}$, $V_{<0}$ and $V_{=0}$. Let also 
\begin{align*}
	V(s) = \norm{v(\cdot,s)}_{W} \text{ and } \delta(s) = \norm{v(\cdot,s)}_{C^{2,\alpha}}.
\end{align*}
From now on we assume that $v$ satisfies the assumptions of \cref{strong-uniqueness}, i.e. \cref{smooth-convergence} and \cref{negative-dominant}. Using the asymptotic expansion of $Q$, \cref{asymptotic-expansion-q}, we see that, up to a time translation, on $(-\infty,0)$,
\begin{align}
	\label{a-1}
	\norm{(\frac{\partial}{\partial s} - L_\Sigma)v(\cdot,s)}_{W} \le C\delta(s)\norm{v(\cdot,s)}_{W^1}
\end{align}
The gradient flow structure of the relative expander entropy implies that $E_{\rel}^*[\Sigma_v,\Sigma] \le 0$. Moreover, by the reverse Poincar\'{e} inequality, \cref{reverse-poincare},
\begin{align*}
	0 \ge E_{\rel}^*[\Sigma_v,\Sigma] \ge C_1 \int_\Sigma \abs{\nabla_\Sigma v}^2 \wt - C_2 \int_{\Sigma} \abs{v}^2 \wt,
\end{align*}
so that \cref{a-1} implies 
\begin{align*}
	\norm{(\frac{\partial}{\partial s} - L_\Sigma)v(\cdot,s)}_{W} \le C\delta(s)\norm{v(\cdot,s)}_{W}.
\end{align*}
Knowing this, we can apply the projection operator $\Pi_{\sim \mu}$ to \cref{nonlinear-problem} and obtain, following \cite{ADS}, for each $\mu \in \{\lambda_1,\ldots,\lambda_I\} \cup \{0\}$, the system: 
\begin{gather}
	\frac{d}{ds}V_{>\mu}  + \bar{\mu} V_{> \mu} \le \delta(s) V,  \label{plus-mode}\\
	\abs{\frac{d}{ds} V_{=\mu} + \mu V_{=\mu}} \le C\delta(s)V,  \label{neutral-mode}\\
	\frac{d}{ds}V_{<\mu}+ \underbar{$\mu$}V_{<\mu}  \ge -\delta(s)V. \label{minus-mode}
\end{gather}
Here $\bar{\mu}$ is the smallest eigenvalue above $\mu$ and \underbar{$\mu$} is the largest eigenvalue below $\mu$. \par 
When $\eps$ is sufficiently small and $v$ not the trivial solution, we may apply \cref{ode-lemma} after multiplying $e^{\mu s}$ to the system to obtain that 
\begin{align*}
	V_{>\mu}(s) \le C\delta(s)(V_{=\mu}(s) + V_{<\mu}(s)) \text{ for } s \in (-\infty,0],
\end{align*}
and that either there exists $s_0 \in (-\infty,0)$ such that $V_{<\mu}(s) \le C \delta(s)V_{=\mu}(s)$ on $(-\infty, s_0]$, or $V_{=\mu}(s) \le C\delta(s)V_{<\mu}(s)$ on $(-\infty,0]$. \cref{negative-dominant} means that the second case happens when $\mu = 0$. \par 
We claim that $\delta(s) \le C e^{-\lambda_{I} s}$. To see this, by \cref{minus-mode}, for every $\eps > 0$ there exists $s_\eps < 0$ such that
\begin{align*}
	\frac{d}{ds}\log V_{-} \ge -\lambda_{I} - C\delta(s) \ge -\lambda_{I} - \eps \text{ on } (-\infty, s_\eps)
\end{align*}
as $\delta(s) \to 0$ as $s \to -\infty$. Integrating this gives the pointwise bound 
\begin{align*}
	V_{-}(s) \le 	V_{-}(s_\eps) e^{(-\lambda_{I} -\eps) (s - s_\eps)},
\end{align*}
for $s < s_\eps$, which, together with the interior Schauder estimates \cite[Theorem C.2]{ChoiMantoulidis}, implies the pointwise decay (note that here it suffices to use the usual estimate for $L^2(\Sigma)$, as $W(\Sigma) \subset L^2(\Sigma)$)
\begin{align*}
	\delta(s) \le C_\eps e^{(-\lambda_{I} - \eps) s }.
\end{align*}
Going back to \cref{minus-mode} and multiplying both sides by $e^{\lambda_{I} s}$ implies that 
\begin{align*}
	\frac{d}{ds} \log(e^{\lambda_{I} s} V_{-}(s) ) \ge -C\delta(s) \ge C_{\eps} e^{(-\lambda_{I} - \eps)s}.
\end{align*}
Integrating this from $s$ to 0 gives that 
\begin{align*}
	\log(e^{\lambda_{I} s} V_{-}(s) )  \le C + V_{-}(0) \implies V_{-}(s) \le Ce^{-\lambda_{I}s}.
\end{align*}
Finally, interior Schauder estimates give that $\delta(s) \le C e^{-\lambda_{I}s}$ as desired. 

\begin{prop}
	Suppose $\mu \in \{\lambda_1,\ldots,\lambda_I\}$ is such that
	\begin{align}
		\label{decay-condition}
		\lambda_i \ge \mu \implies \lim_{s \to -\infty} e^{\lambda_i s} \inner{v(\cdot,s), \phi_i} \phi_i = 0.
	\end{align}
	Then $\mu \ne \lambda_1$ and $V_{\ge \mu}(s) \le C \delta(s) V_{<\mu}(s)$ for all $s \le 0$.
\end{prop}
\begin{proof}
	We proceed by induction. Suppose $\mu = \lambda_I$, and suppose for a contradiction that there is $s_0$ such that $V_{<\lambda_I}(s) \le C\delta(s)V_{=\lambda_I}(s)$ on $(-\infty,s_0]$. Multiplying \cref{neutral-mode} by $e^{\lambda_I s}$ gives that
	\begin{align*}
		|\frac{d}{ds} (e^{\lambda_I s} V_{=\lambda_I})| \le C e^{\lambda_I s} \delta(s) V(s) \le C e^{\lambda_I s} \delta(s) V_-(s) \le Ce^{\lambda_I s} \delta(s) V_{=\lambda_I}(s),
	\end{align*}
	where we used \cref{negative-dominant}. This implies that 
	\begin{align*}
		\abs{\frac{d}{ds} \log (e^{\lambda_I s} V_{=\lambda_I}(s))} \le C\delta(s) \text{ on } (-\infty,s_0]
	\end{align*}
	As $\delta(s) \le Ce^{-\lambda_I s}$, we see that this gives a contradiction upon integrating. Thus the alternative must hold, i.e. 
	\begin{align*}
		V_{=\lambda_I}(s) \le C\delta(s) V_{<\lambda_I}(s)  \text{ on } (-\infty,0].
	\end{align*}
	In particular this means that $\mu \ne \lambda_1$. In general, suppose the proposition is true for $\mu = \lambda_J$, then for $\mu = \lambda_{J'}$ the largest eigenvalue below $\lambda_J$, we can repeat the above argument using the fact that 
	\begin{align*}
		V_-^2 = V_{0 > \mu > \lambda_J'}^2 + V_{= \lambda_J'}^2 + V_{< \lambda_J'}^2,
	\end{align*}
	and the proof follows verbatim until $\lambda_{J'} = \lambda_1$.
\end{proof}
Let $I^*$ be the largest index for which \cref{decay-condition} fails, then it follows, from \cref{minus-mode} applied to the smallest eigenvalue above $I^*$, that
\begin{align}
	\label{improved-minus-mode}
	\frac{d}{ds} V_{\le \lambda_{I^*}} + \lambda_{I^*} V_{\le \lambda_{I^*}} \ge -C\delta(s)V_{\le \lambda_{I^*}}.
\end{align}
Using \cref{improved-minus-mode} in place of \cref{minus-mode}, a similar argument as above shows that  $\delta(s) \le C e^{-\lambda_{I^*} s}$, which is the sharp asymptotic decay in time for the solution. \par 
To finish the proof, we seek to apply the uniqueness aspect of \cref{ancient-solution-existence}. It suffices to establish \cref{closeness-estimate}. For $\sigma > 0$, let $v^{(\sigma)}(p,s) = v(p,s- \sigma)$ be the translated solution. Then by definition of $I^*$ we have 
\begin{align*}
	\limsup_{\sigma \to \infty} e^{-\lambda_{I^*} \sigma} \norm{\Pi_{-} v^{(\sigma)}(\cdot,0)}_{W} > 0.
\end{align*}
\begin{prop}
	\label{prop-b-5}
	Given $\sigma \ge 0$, then it holds for $s \le 0$,
	\begin{align*}
		e^{-2\lambda_{I^*} \sigma} ||v^{(\sigma)}(\cdot, s) - \sum_{i \le I} e^{-\lambda_i (s - \sigma)} \inner{v(\cdot, s - \sigma), \phi_i} \phi_i||_{W} \le C e^{-\lambda_I s}.
	\end{align*}
\end{prop}
\begin{proof}
	As $\delta(s) \le Ce^{-\lambda_{I^*} s}$ and the negative mode is dominant, we have 
	\begin{align*}
		e^{-2\lambda_{I^*}\sigma} (V_0 + V_+) (s - \sigma) \le C e^{-2\lambda_{I^*}s} \le C e^{-\lambda_I s}. 
	\end{align*}
	On the other hand, \cref{minus-mode} with $\mu = 0$ implies that for every $1 \le i \le I$,
	\begin{align*}
		||\frac{d}{ds} u_i+ \lambda_i u_i||_{W} \le C\delta(s) V_{\le I^*},
	\end{align*}
	where $u_i = e^{-\lambda_i (s-\sigma)} \inner{v(\cdot,s-\sigma), \phi_i}\phi_i$. Multiplying this equation by $e^{\lambda_i s}$ and integrating from $s-\sigma$ to $-\sigma$ yields, 
	\begin{align*}
		\norm{e^{\lambda_i (s-\sigma)}u_i(\cdot, s-\sigma) - e^{-\lambda_i \sigma} u_i(-\sigma)}_{W} \le C \int_{s-\sigma}^{-\sigma} e^{\lambda_i \rho} e^{-2\lambda_{I^*} \rho} d\rho.
	\end{align*}
	Thus 
	\begin{align*}
		&\phantom{{}\le{}} e^{-2\lambda_{I^*} \sigma} \norm{u_i(\cdot, s-\sigma) - e^{-\lambda_i s} u_i(-\sigma)}_{W} \\
		&\le Ce^{-2\lambda_{I^*}\sigma} \int_{s-\sigma}^{-\sigma} e^{-\lambda_i (s-\sigma - \rho)} e^{-2\lambda_{I^*} \rho} d\rho \\
		&\le C\abs{s} e^{-2\lambda_{I^*} s} \le C e^{-\lambda_I s},
	\end{align*}
	provided $\sigma$ is sufficiently large. 
\end{proof}
Now let $0 < \delta < -\lambda_I$. Using interior Schauder estimates \cite[Theorem C.2]{ChoiMantoulidis} and \cref{main-schauder-estimates} together with \cref{prop-b-5}, we see that, for every $s \le 0$,
\begin{align*}
	e^{(\lambda_I+\delta) s} ||v^{(\sigma)}(\cdot, s) - \sum_{i \le I} e^{-\lambda_i (s - \sigma)} \inner{v(\cdot, s - \sigma), \phi_i} \phi_i||_{C^{2,\alpha}_P(\Sigma \times [s-1,s])} \le Ce^{2\lambda_{I^*} \sigma}.
\end{align*}
\cref{strong-uniqueness} now follows from \cref{ancient-solution-existence} after choosing $\beta$ sufficiently large depending on $C$ and then $\sigma$ sufficiently large so that $\abs{\inner{v(\cdot,-\sigma),\phi_i}} < \eta$, where $\eta = \eta(\beta)$. \par 
It is now simple to deduce \cref{main-uniqueness-theorem}. Recall from \cref{backwards-convergence} that, up to a time translation, we may assume that $\tilde{\mathcal{M}}$ can be written as a normal graph of $v$ over $\Sigma$ on $(-\infty,0)$ such that
\begin{align*}
	\norm{v}_{C^{2,\alpha}_P(\Sigma \times (-\infty,0])} < \eps. 
\end{align*}
This gives \cref{smooth-convergence}. As $\mathcal{C}$ is generic, $\Sigma$ has no nontrivial Jacobi field, and therefore \cref{negative-dominant} holds (if the neutral or the positive modes were to be dominant, the solution must be static). The rest follows from \cref{strong-uniqueness}. \par 
We now discuss briefly the converse question; namely when is a tame ancient RMCF a Morse flow line. In general, one cannot extend an ancient RMCF to an eternal one. In $\R^3$, using the classification of low entropy self-shrinkers of Bernstein and Wang \cite{BWTopology}, we have the following strong converse to \cref{ancient-solution-existence} in the low entropy setting. Recall that the entropy of a hypersurface $\Sigma$, as defined in Colding--Minicozzi \cite{CMGeneric}, is given by
\begin{align*}
	\lambda[\Sigma] = \sup_{x_0 \in \R^{n+1}, t > 0} (4\pi t)^{-\frac{n}{2}} \int_\Sigma e^{\frac{\abs{\xX(p) - x_0}^2}{4t}} d\mathcal{H}^n.
\end{align*}
\begin{cor}
	\label{low-entropy-flow-lines}
	Suppose $\mathcal{C} \subset \R^3$ is a cone with $\lambda[\mathcal{C}] < \lambda[\mathbb{S} \times \mathbb{R}]$. Then any tame ancient RMCF is a Morse flow line. In particular, there exist an $I$-parameter family of Morse flow lines coming out of an index $I$ self-expander asymptotic to $\mathcal{C}$.
\end{cor}
\begin{proof}
	By Huisken's monotonicity formula, any singularity of the flow must have entropy less than $\lambda[\mathbb{S} \times \R]$. By \cite[Corollary 1.2]{BWTopology}, it must be a round sphere $\mathbb{S}^2$. However, as any tame ancient RMCF is asymptotically conical (as $\Sigma$ is asymptotically conical), it cannot encounter a compact singularity at the first singular time. Thus, any such flow must remain smooth for all time. The second conclusion follows in view of \cref{main-theorem}.
\end{proof}
\begin{rem}
	We suspect that the entropy bound can be relaxed, with a suitable surgery procedure, to $\lambda[\mathcal{C}] < \lambda[\mathbb{S} \times \R] + \delta$. 
\end{rem}
In a different direction, if we know that the flow is expander mean convex at some time, it is also possible to extend an ancient RMCF to a Morse flow line. Recall that a hypersurface $\Sigma$ is expander mean convex if 
\begin{align*}
	\hH_\Sigma(p) + \frac{1}{2} \nN_\Sigma(p) \cdot \xX(p)  > 0.
\end{align*}
In view of \cref{nonlinear-problem}, a graphical RMCF over a self-expander $\Sigma$ is expander mean convex if and only if $v > 0$; that is, the RMCF lies on one side of $\Sigma$. An expander mean convex RMCF stays expander mean convex for all future time as long as it is smooth. In \cite{BCW}, we extended the notion of expander mean convexity past singularities. In particular, we showed that a smooth expander mean convex RMCF can be extended in some appropriate weak sense to stay expander mean convex in all future time, regardless of singularities, and that the extended flow is a rescaled Brakke flow. Moreover, the forward limit of such a flow is always a stable self-expander. This limit is unique due to expander mean convexity, and is smooth in low dimensions. \par 
Using a slight modification of the argument used in \cite{BCW}, we have the following partial converse to \cref{strong-uniqueness}.
\begin{prop}
	\label{existence-flow-line}
	Let $2 \le n \le 6$. If $v$ is an ancient solution constructed in \cref{ancient-solution-existence} and there is $s_0 < 0$ such that $v(\cdot, s_0) > 0$, then $v$ can be extended to a Morse flow line.  
\end{prop}
\begin{proof}
	Since $v(\cdot, s_0) > 0$,  by the strong maximum principle (see eg. \cite{Smoczyk}), we have $v > 0$ on $(s_0,0]$. We can then follow the construction in \cite[Section 3]{BCW} to extend the flow in an expander mean convex way past the singularity (the results there are stated for perturbations of the first eigenfunction $\phi_1$, but the proof only uses the fact that $\phi_1$ has a sign).  Since we are in low dimensions, the limiting self-expander $\Gamma$ must be smooth, and so $v$ gives rise to a Morse flow line between $\Sigma$ and $\Gamma$.
\end{proof}
The prototypical example of an expander mean convex RMCF is the ancient RMCF corresponding to $a_1 \phi_1$ in \cref{ancient-solution-existence}, which is expander mean convex for all time. However, the above proposition is not so effective as it is the condition $v(\cdot,s) > 0$ is hard to check. Indeed, the first eigenfunction $\phi_1$ is the only eigenfunction that has a sign, but at the same time it also has the best asymptotic decay among all eigenfunctions. In fact, an ambitious conjecture would be $v(\cdot, s_0) > 0$ for some $s_0 \in (-\infty,0)$ implies the same for all $s_0 < 0$.  \par 
We end the section by recording the following uniqueness theorem for expander mean convex RMCFs asymptotic to a generic cone, which might be of independent interest.
\begin{prop}
	\label{1-sided-flow}
	Suppose $\mathcal{C}$ is generic, then, up to time translation, there is a unique solution $v$ to \cref{nonlinear-problem} such that $v > 0$ on $(-\infty,0)$. 
\end{prop}
\begin{proof}
	Let $v > 0$ be a positive solution to $\cref{nonlinear-problem}$ on $(-\infty,0)$. As $\mathcal{C}$ is generic, we can follow the spectral analysis above to get the sharp asymptotic decay
	\begin{align}
		\label{sharp-asymptotics}
		\delta(s) \le Ce^{-\lambda_{I_*}s}.
	\end{align}
	We claim that $\lambda_{I_*} = \lambda_1$.  This follows from the fact that
	\begin{align*}
		V_{\ne \lambda_{I_*}}(s) \le C \delta(s) V_{=\lambda_{I_*}} 
	\end{align*}
	when $s \le s_0$ for some $s_0 < 0$ (indeed if the above were true, it will violate the sharpness of the estimate \cref{sharp-asymptotics}). To see this, note that $v > 0$ implies that
	\begin{align*}
		-\min\{0, \Pi_{=\lambda_{I_*}} v\} \le \abs{\Pi_{\ne \lambda_{I_*} }v} \implies \norm{-\min\{0, \Pi_{=\lambda_{I_*}}v\}} \le C\delta(s)  V_{=\lambda_{I_*}} 
	\end{align*}
	Now let $h^{s} =  V_{=\lambda_{I_*}}^{-1}(s) \Pi_{=\lambda_{I_*}} v(\cdot, s)$, then as $s \to -\infty$, $h^{s}$ converges to an $\lambda_{I_*}$-eigenfunction $h$ of norm 1. Since $\delta(s) \to 0$ as $s \to -\infty$, it follows that
	\begin{align*}
		\norm{-\min\{0,h\}}_W = 0,
	\end{align*}
	and so $h \ge 0$. By standard spectral theory the only eigenfunction that does not change sign corresponds to the lowest eigenvalue. Hence $\lambda_{I_*} = \lambda_1$. In particular, for any positive solution $v$, there exists a constant $\alpha_1 \ne 0$ such that 
	\begin{align*}
		\lim_{s \to -\infty} e^{s \lambda_1} v(\cdot, s) = \alpha_1 \phi_1.
	\end{align*}  \par
	Since $\mathcal{C}$ is generic, we can apply the strong uniqueness theorem \cref{strong-uniqueness} to conclude that $v = \Phi(a)$ for some $a = (a_1,\ldots,a_I) \in \R^I$. By \cref{closeness-estimate}, we have that, for $s \le 0$,
	\begin{align*}
		\norm{v - \sum_{i=1}^I a_i e^{- \lambda_i s} \phi_i}_{C^{2,\alpha}_P(\Sigma \times [s-1,s])} \le e^{\delta_0 s} \beta \sum_{i=1}^I a_i^2
	\end{align*}
	where $0 < \delta_0 < -\lambda_I$ and $\beta > 0$. Multiplying by $e^{\lambda_1 s}$ on both sides yields
	\begin{align*}
		\norm{e^{s\lambda_1}v - \sum_{i=1}^I a_i e^{(-\lambda_i + \lambda_1) s} \phi_i}_{C^{2,\alpha}_P(\Sigma \times [s-1,s])}\le e^{\delta_0 + \lambda_1 s} \beta \sum_{i=1}^I a_i^2
	\end{align*}
	Since $v(\cdot, s) \le Ce^{-\lambda_1 s}$ and $\delta_0 +\lambda_1 < \lambda_1 - \lambda_I \le \lambda_1 - \lambda_i$ for all $i \ge 1$, the above can only hold if $a_i = 0$ for all $2 \le i \le I$. This proves that there is a one-parameter family of positive solutions to \cref{nonlinear-problem}, which corresponds precisely to time translations.
\end{proof}
\begin{rem}
	By modifying a beautiful iteration argument of Chodosh--Choi--Mantoulidis--Schulze \cite[Corollary 5.2]{CCMSGeneric}, we expect that the uniqueness continues to hold without the genericity assumption (essentially, the nonlinear term $Q$ in the expander case satisfies the same estimates as in the shrinker case). As our article mostly concerns with generic cones, we have chosen to state the simpler version of the uniqueness result.
\end{rem}
\section{Mean curvature flows}
\label{cone-section}
All of our analysis so far has been on the level of RMCF, and essentially only relies on the fact that any such flow, up to a time translation, is an entire graph over the self-expander $\Sigma$. In this section, we study mean curvature flows coming out of cones and indicate when such a flow fits into our discussion above. We will work in low dimensions, i.e. $2 \le n \le 6$., where the structure theory of self-expanders is best known, thanks to a series of works of Bernstein and Wang \cite{BWSpace, BWSmoothCompactness, BWMountainPass, BWRelativeEntropy}. \par 
We say an asymptotically conical hypersurface $\Gamma$ is \textit{trapped} if there exist two self-expanders $\Sigma_1$ and $\Sigma_2$ asymptotic to $\mathcal{C}$ such that there is some radius $R_0 > 0$ such that 
\begin{align*}
	\Gamma \setminus B_{R_0}(0) \subset \Omega_1^+ \cap \Omega_2^-,
\end{align*}
where $\Omega_1^{\pm}, \Omega_2^{\pm}$ denote the connected components of $\mathbb{R}^{n+1} \setminus \Sigma_1$ and $\mathbb{R}^{n+1} \setminus \Sigma_2$ respectively, oriented in a way such that $\Omega_1^+ \subset \Omega_2^+$. In the special case $\Gamma = \Sigma_v$, using the asymptotic structure of self-expanders at infinity \cite{Bernstein}, we deduce that $v$ satisfies the sharp decay estimates
\begin{align}
\label{trapped-decay}
	\abs{v(p)} \le C\abs{\xX(p)}^{-n-1}e^{-\frac{\abs{\xX(p)}^2}{4}} \text{ when } \abs{\xX(p)} > R_0
\end{align}
First of all, under the trapping assumption, we can prove that the relative expander entropy $E_{\rel}^*$ defined in \cref{uniqueness-section} indeed coincides with the usual notion $E_{\rel}$ from \cite{BWRelativeEntropy}. This seemingly innocent fact has to do with the failure of the normal graph map $\fF_v$ being bijective in the annulus $B_{R+2}(0) \setminus B_R(0)$. Due to the large weight $\wt$, this difference cannot be killed unless the function $v$ has a very good decay.
\begin{prop}
	\label{entropy-equivalence}
	Suppose $v: \Sigma \to \R$ is a $W^1$ function such that $E_{\rel}^*[\Sigma_v,\Sigma] < \infty$. If $\Sigma_v$ is trapped between two self-expanders $\Gamma_1,\Gamma_2$ asymptotic to $\mathcal{C}$ then 
	\begin{align*}
		E_{\rel}[\Sigma_v,\Sigma] = E_{\rel}^*[\Sigma_v,\Sigma] < \infty.
	\end{align*}
\end{prop}
\begin{proof}
	We fix $R$ sufficiently large and write $\chi_R = \chi$ for simplicity. Let $A_{r,R} = B_{R}(0) \setminus \bar{B}_r(0)$. Since $\fF_{\chi v}$ is an embedding, it is a diffeomorphism between $\Sigma_{\chi v} \cap B_{R+2}(0)$ and $\Sigma \cap B_{R+2}(0)$. Let $Y_R = \fF^{-1}_{\chi v}(\Sigma_{\chi v} \cap B_{R}(0)) \subset \Sigma\cap B_{R+2}(0)$ and $Z_R = (\Sigma \cap B_{R+2}(0) \setminus Y_R)$. By the triangle inequality 
	\begin{align*}
		\abs{\fF_{\chi v}(p)} \le \abs{\xX(p)} + \abs{v(p)},
	\end{align*}
	consequently the sets $Y_R$ satisfy 
	\begin{align}
		Y_R \Delta (\Sigma \cap B_R(0)) \subset \Sigma \cap A_{R - \bar{v}_R, R+\bar{v}_R}
	\end{align}
	where $\bar{v}_R = \sup_{\Sigma \cap A_{R,R+2}} \abs{v}$, and $\Delta$ denotes the symmetric difference of two sets. Since $\Sigma$ is trapped, it follows from \cref{trapped-decay} that
	\begin{align*}
		\abs{Y_R \Delta (\Sigma \cap B_R(0))}= O(R^{-2} e^{\frac{-R^2}{4}}).
	\end{align*}
	Now write 
	\begin{align*}
		\int_{\Sigma_{\chi v}} \wt - \int_{\Sigma} \wt &= \left(\int_{\Sigma_{v} \cap B_{R}(0)} \wt - \int_{Y_R} \wt\right) \\
		&\phantom{{}={}}+ \left(\int_{\Sigma{\chi v} \cap A_{R,R+2}} \wt - \int_{Z_R} \wt\right).
	\end{align*}
	We can estimate
	\begin{align*}
		\abs{\int_{Y_R} \wt - \int_{\Sigma \cap B_R(0)} \wt} \le \int_{Y_R \Delta (\Sigma \cap B_R(0))} \wt = O(R^{-2}),
	\end{align*}
	where we used the fact that, when $R$ is sufficiently large
	\begin{align*}
		\wt \le e^\frac{(R+\bar{v}_R)^2}{4} \le e^{\frac{R^2}{4}} \left(1 + \bar{v}_R R + 2\bar{v}_R^2\right) \le e^{\frac{R^2}{4}} + CR^{-n} \le e^{\frac{R^2}{4}} + 1.
	\end{align*}
	Thus 
	\begin{align}
		\label{eq-3-1}
		\lim_{R\to \infty} \int_{Y_R} \wt - \int_{\Sigma \cap B_R(0)} \wt = 0.
	\end{align}
	Moreover, using \cref{expansion-1}, we get that
	\begin{align*}
		&\phantom{{}\le{}}\int_{\Sigma{\chi v} \cap A_{R,R+2}} \wt - \int_{Z_R} \wt \\
		&\le \frac{1}{2} \int_{Z_R} (\abs{\nabla_\Sigma v}^2 + \abs{\nabla_\Sigma \chi} \chi(\abs{\nabla_\Sigma v}^2 + v^2) + (\frac{3}{2} - \abs{A_\Sigma}^2)v^2 + C\eps(\abs{\nabla_\Sigma v}^2 + v^2)) \wt.
	\end{align*}
	As $v \in W^1$,
	\begin{align*}
		\lim_{R\to \infty} \int_{Z_R} (\abs{\nabla_\Sigma v}^2 + v^2)\wt = 0.
	\end{align*}
	Since $\abs{A_\Sigma}$ is bounded, the above implies that 
	\begin{align}
		\label{eq-3-2}
		\lim_{R\to \infty} \int_{\Sigma{\chi v} \cap A_{R,R+2}} \wt - \int_{Z_R} \wt = 0
	\end{align}
	\cref{eq-3-1} and \cref{eq-3-2} imply the desired equality.
\end{proof}
The above equivalence means that we have all the tools from \cite{BWRelativeEntropy} in our disposal, and from now on we will write unambiguously $E_{\rel}$ for the relative entropy. \par 
Now let $\mathcal{M} = \{\mu_t\}_{t \in (0,T]}$ be an integral Brakke flow coming out of $\mathcal{C}$ in the sense that 
\begin{align*}
	\lim_{t \to 0} \mu_t = \mathcal{H}^n \llcorner \mathcal{C}.
\end{align*}
$\mathcal{M}$ is contained in the level set flow of $\mathcal{C}$ (which necessarily fattens as long as there are more than one self-expanders asymptotic to $\mathcal{C}$). The key fact is that, when $2 \le n \le 6$, by \cite[Theorem 8.21]{CCMSGeneric}, the two outermost flows of the cone $\mathcal{C}$ corresponding to the boundary of the level set flow are given by two stable self-expanders (which are smooth when $2 \le n \le 6$). Hence $\mathcal{M}$ is, in fact, trapped between two asymptotically conical self-expanders. In particular the following forward monotonicity formula holds (here we have taken $f = 1$):
\begin{prop}[Proposition 6.5 of \cite{BWRelativeEntropy}]
\label{forward-monotonicity}
	Let $\mathcal{M}$ be as above, and let $\tilde{\mathcal{M}} $ denote the corresponding rescaled Brakke flow defined on $(-\infty, S)$. Then for any $-\infty < s_1 \le s_2 \le S$, we have 
	\begin{align*}
		E_{\rel}[\tilde{\mu}_{s_1}, \Sigma] \ge E_{\rel}[\tilde{\mu}_{s_2},\Sigma] + \int_{s_1}^{s_2}  \abs{\hH_{\tilde{\mu}_s} - \frac{\xX^\perp}{2}}^2 \wt d\tilde{\mu}_s ds.
	\end{align*}
\end{prop}
The following backward convergence to a self-expander is the starting point of the analysis. 
\begin{prop}
	\label{backward-limit}
	Let $\mathcal{M}$ be as above. Up to passing to a subsequence $s_i \to -\infty$, there exists a self-expander $\Sigma$ such that
	\begin{align*}
		\lim_{i \to \infty} \tilde{\mu}_{s_i} = \mathcal{H}^n \llcorner \Sigma.
	\end{align*}
\end{prop}
\begin{proof}
	Fix a reference self-expander $\Sigma'$. Then $E_{\rel}[\tilde{\mu}_s, \Sigma']$ is finite by appealing to the trapping and \cite{BWRelativeEntropy}. Consider the translated flow $\tilde{\mu}_{s}^{s_0} = \tilde{\mu}_{s + s_0}$ defined on $s \in (-\infty,S-s_0)$. By compactness of Brakke flows, up to passing to a subsequence $\{s_i\}$, as $s_i \to -\infty$, the sequence of translated flows converges to a rescaled Brakke flow $\bar{\mu}$ defined on $(-\infty,\infty)$. Moreover, 
	\begin{align*}
		E_{\rel}[\bar{\mu}_s, \Sigma'] = \lim_{i \to \infty}  E_{\rel}[\tilde{\mu}_{s}^{s_i}, \Sigma'] < \infty.
	\end{align*}
	In view of the forward monotonicity formula \cref{forward-monotonicity}, we have 
	\begin{align*}
		\int_{s_1}^{s_2}\int \abs{\hH_{\bar{\mu}_s} - \frac{\xX^\perp}{2}}^2 d\bar{\mu}_s ds = 0
	\end{align*}
	for all $-\infty < s_1 < s_2 < \infty$. Hence $\bar{\mu}$ is $E$-stationary, and therefore a self-expander. Moreover as $\tilde{\mathcal{M}}$ is asymptotic to $\mathcal{C}$, so is $\bar{\mu}$.
\end{proof}
In general, however, $\Sigma$ is only an $E$-stationary varifold and can have very large singular sets even in low dimensions. This is, in some sense, the key obstruction in the correspondence between tame ancient RMCFs and MCFs coming out of $\mathcal{C}$. One way to resolve this issue is to impose a low entropy condition, which forces the expanders to be smooth. \par
Another powerful consequence of the trapping is the uniqueness of tangent flows, which is proved using the \L ojasiewicz inequality. Let $\mathcal{N}_\Sigma$ be the Euler-Lagrange operator associated with $E_{\rel}[\cdot,\Sigma]$ given by 
\begin{align*}
	\left.\frac{d}{ds}\right|_{s = 0} E_{\rel}[\Sigma_{v +sw}, \Sigma] = \inner{\mathcal{N}_\Sigma v, w}.
\end{align*}
As $0$ is a critical point for $E_{\rel}$, the operator $\mathcal{N}_\Sigma$ takes the form
\begin{align*}
	\mathcal{N}_\Sigma v = L_\Sigma v + Q(v),
\end{align*}
which, in fact, agrees with the expander mean curvature of the hypersurface $\Sigma_v$. Here we record a version of the \L ojasiewicz inequality for generic cones, which is good enough for our applications. Since we do not use the trapping assumption in the following proof, the same argument will work for $E_{\rel}^*$ if we do not know a priori that $E_{\rel}$ and $E_{\rel}^*$ agree. However, in the proof of the uniqueness of tangent flows, \cref{uniqueness-of-tangent}, trapping is necessary. \par 
\begin{thm}
	\label{lojasiewicz-relative-entropy}
	Let $\Sigma$ be a self-expanders asymptotic to a generic cone $\mathcal{C}$. There is $\eps = \eps(\Sigma)$ such that the following holds: suppose $v \in C^{2,\alpha}\cap W^2(\Sigma)$ satisfies $\norm{v}_{C^{2,\alpha}}  < \eps$, then 
	\begin{align}
	\label{lojasiewicz}		
	C\norm{\mathcal{N}_\Sigma(v)}_{W} \ge \abs{E_{\rel}[\Sigma_v,\Sigma]}^{1/2}.
	\end{align}
\end{thm}
\begin{proof}
Let $\mathcal{L}_\Sigma$ denote the operator
\begin{align*}
	\mathcal{L}_\Sigma v = \Delta_\Sigma v + \frac{1}{2}\xX \cdot \nabla_\Sigma v - \frac{1}{2}v = L_\Sigma v - \abs{A_\Sigma}^2 v.
\end{align*}
By \cite[Proposition 3.4]{BWIntegerDegree}, $\mathcal{L}_\Sigma $ is an isomorphism between $W^2$ and $W$, so there exists a constant $C = C(\Sigma)> 0$ such that
\begin{align*}
	\norm{v}_{W^2} \le C\norm{\mathcal{L}_\Sigma v}_W \le C\norm{L_\Sigma v}_W + CC'\norm{v}_W,
\end{align*}
where $C' = C'(\Sigma) >0$ depends only the bound on $\abs{A_\Sigma}^2$. As $\mathcal{C}$ is generic, we have 
\begin{align*}
	\norm{L_\Sigma v}_W \ge c \norm{v}_W
\end{align*}
for some constant $c > 0$ depending on the spectral gap $\bar{\lambda} = \min\{\abs{\lambda_i}\} > 0$. Thus 
\begin{align*}
	\norm{v}_{W^2} \le C\norm{L_\Sigma v}_W + c^{-1}CC' \norm{L_\Sigma v}_W \implies C_\Sigma\norm{v}_{W^2} \le \norm{L_\Sigma v}_W,
\end{align*}
where the constant $C_\Sigma > 0$ in the last inequality depends only on $\Sigma$. Using the expansion \cref{asymptotic-expansion-q}, there is $\eps > 0$ such that when $\norm{v}_{C^{2,\alpha}} < \eps$, we have
\begin{align*}
	\norm{Q(v,\nabla_\Sigma v, \xX \cdot \nabla_\Sigma v,\nabla_\Sigma^2 v)}_W \le \frac{1}{2} C_\Sigma \norm{v}_{W^1} \le \frac{1}{2} C_\Sigma\norm{v}_{W^2}.
\end{align*}
Hence it follows from the triangle inequality that
\begin{align}
\label{eq-5-4}
	\norm{\mathcal{N}_\Sigma v}_W = \norm{L_\Sigma v + Q(v)}_W \ge \frac{1}{2} C_\Sigma \norm{v}_{W^2}
\end{align}
On the other hand, \cref{motivating-computation} implies that
\begin{align}
\label{eq-5-5}
	\abs{E_{\rel}[\Sigma_v,\Sigma]} \le C_0 \norm{v}_{W^1}^2 \le C_0\norm{v}_{W^2}^2
\end{align}
where $C_0 = C_0(\Sigma) > 0$. Combining \cref{eq-5-4} and \cref{eq-5-5}, we see that
\begin{align*}
	4C_0C_\Sigma^{-2} \norm{\mathcal{N}_\Sigma v}^2_W \ge C_0 \norm{v}_{W^2}^2 \ge \abs{E_{\rel}[\Sigma_v,\Sigma]}
\end{align*}
holds whenever $\norm{v}_{C^{2,\alpha}} < \eps$.
\end{proof}
\begin{rem}
	Here we shall explain that the exponent $\gamma = 1$ in \cref{better-lojasiewicz} is the best possible by showing \cref{lojasiewicz} implies \cref{better-lojasiewicz} when $\Sigma_v$ is trapped. It is enough to show
	\begin{align}
		\label{eq-5-2}
		\norm{\mathcal{N}_\Sigma(v)}_W^{\gamma_2} \le C^{\gamma_2 - \gamma_1} \norm{\mathcal{N}_\Sigma(v)}_W^{\gamma_1},
	\end{align}
	whenever $\gamma_1 \le \gamma_2$ and $\norm{v}_{C^{2,\alpha}} < \eps$ for $\eps$ sufficiently small. We point out that, in the following argument, the trapping assumption is also essential. \par 
	As $\mathcal{L}_\Sigma$  is an isomorphism between $W^2$ and $W$,
	\begin{align*}
		\norm{L_\Sigma v}_W^2 \le C \norm{v}_{W^2}^2.
	\end{align*}
	 Together with the expansion of $Q(v)$, \cref{asymptotic-expansion-q}, we conclude that 
	\begin{align*}
		\norm{\mathcal{N}_\Sigma(v)}_W^2 \le C \norm{v}_{W^2}^2.
	\end{align*}
	Given $\delta > 0$, we can choose $R_1 > R_0$ depending on $\delta$ such that
	\begin{align*}
		\int_{\Sigma \setminus B_{R_1}(0)} \abs{v}^2 \wt \le \frac{\delta}{2},
	\end{align*}
	where we used the sharp decay rate of $v$ in \cref{trapped-decay}. As $\Sigma \cap B_{R_1}(0)$ is compact, we can find $\eps = \eps(\delta, \Sigma)$ such that
	\begin{align*}
		\int_\Sigma \abs{v}^2 \wt = \int_{\Sigma \cap B_{R_1}(0)} \abs{v}^2 \wt + \int_{\Sigma \cap B_{R_1}(0)} \abs{v}^2 \wt \le \delta
	\end{align*}
	whenever $\norm{v}_{C^{2,\alpha}} < \eps$. A similar argument shows the same for $\norm{v}_{W^2}$ (possibly shrinking $\eps$).  This means that we can guarantee 
	\begin{align*}
		\norm{\mathcal{N}_\Sigma(v)}_W^2 \le C\delta^2
	\end{align*}
	whenever $\norm{v}_{C^{2,\alpha}} < \eps$. \cref{eq-5-2} immediately follows (by setting $\delta = 1$, for example). 
\end{rem}
Using \cref{lojasiewicz-relative-entropy} we can upgrade subsequential convergence to full convergence in the generic case. Combining \cref{uniqueness-of-tangent} and \cref{main-uniqueness-theorem} gives \cref{main-cor-mcf}. Note that we only require the limit to be smooth for one subsequence in the following.
\begin{cor}
	\label{uniqueness-of-tangent}
	Let $\mathcal{M}$ and $\Sigma$ be as above. If $\Sigma$ is smooth, then
	\begin{align*}
		\lim_{s \to -\infty} \tilde{\mu}_s = \mathcal{H}^n \llcorner \Sigma.
	\end{align*}
\end{cor}
\begin{proof}
	The proof is standard (see eg. \cite{Schulze} or \cite[Chapter 3]{SimonETH}), and we present it for the sake of completeness. For any function $v(\cdot,s) \in C^{2,\alpha} \cap W^2$ with $\norm{v(\cdot,s)}_{C^{2,\alpha}} < \eps$ sufficiently small, we compute, using the forward monotonicity formula \cref{forward-monotonicity}:
	\begin{align*}
		\frac{d}{ds}E_{\rel}[\Sigma_{v(\cdot,s)}, \Sigma] &= -\int_{\Sigma_{v(\cdot,s)}} \abs{\hH_{\Sigma_{v(\cdot,s)}} - \frac{\xX^\perp}{2}}^2 \wt d\mathcal{H}^n \\
		& \le -\left( \int_\Sigma \abs{\mathcal{N}_\Sigma v}^2 \right)^{1/2} \left(\int_\Sigma \abs{\frac{\partial v}{\partial s}}^2 e^{\frac{\abs{\xX + v(x)\nN_\Sigma}^2}{4}} \mathrm{Jac}(v)\right)^{1/2} \\
		& \le C \norm{\mathcal{N}_\Sigma v}_{W} \norm{\frac{\partial v}{\partial s}}_W,
	\end{align*}
	where $C = C(\Sigma)$. Hence, by \cref{lojasiewicz-relative-entropy},
	\begin{align*}
		-\frac{d}{ds} \abs{E_{\rel}[\Sigma_{v(\cdot, s)}, \Sigma]}^{1/2} \ge \frac{C}{2} \abs{E_{\rel}[\Sigma_{v(\cdot,s)},\Sigma]}^{-1/2} \norm{\mathcal{N}_\Sigma v}_W \norm{\frac{\partial v}{\partial s}}_{W} \ge C \norm{\frac{\partial v}{\partial s}}_W
	\end{align*}
	Integrating the above gives, for $s_2 \ge s_1$
	\begin{align}
	\label{integral-bound}
		\int_{s_1}^{s_2} \norm{\frac{\partial v}{\partial s}}_W ds \le C_0 \abs{E_{\rel}[\Sigma_{v(\cdot, s_1)}, \Sigma]}^{1/2},
	\end{align}
	and therefore
	\begin{align}
	\label{eq-5-3}
		\norm{v(\cdot,s_2)}_{W} \le \norm{v(\cdot,s_1)}_W + C_0 \abs{E_{\rel}[\Sigma_{v(\cdot,s_1)},\Sigma]}^{1/2}.
	\end{align}
	for $C_0 = C_0(\Sigma)$. \par
	Let $\eps_1 > 0$ be such that the nearest point projection onto $\Sigma$ is smooth in a tubular neighborhood of radius $\eps_1$. Let $\delta$ be such that the extension lemma \cite[Lemma 2.2]{Schulze} holds with $\beta = \frac{1}{2}$ and $\sigma = \eps_1$ (strictly speaking we need a slightly modified version with $L^2$ norm replaced by the $W$ norm). Choose $\eps_0 = \eps_0(\delta, \Sigma)$ such that $\norm{v}_{C^{2,\alpha}} < \eps_0$ implies $\norm{v}_{W^2} \le \delta/(3C_0)$ (as $\mathcal{M}$ is trapped) and 
	\begin{align*}
		C_0\abs{E_{\rel}[\Sigma_v,\Sigma]}^{\frac{1}{2}} \le 2C_0\norm{v}_{W^2} \le \frac{2\delta}{3}
	\end{align*}
	in view of \cref{motivating-computation}. Now let $s_i \to -\infty$ be a convergent subsequence, i.e. 
	\begin{align*}
		\lim_{i \to \infty} \tilde{\mu}_s = \mathcal{H}^n \llcorner \Sigma.
	\end{align*}
	As $\Sigma$ is a smooth self-expander, by Brakke regularity theorem \cite{Brakke}, for each $i$ there is $\eta_i$ such that $\tilde{\mu}_s$ converges to the static flow of $\Sigma$ on $(s_i - \eta_i, s_i + \eta_i)$. By subsequential convergence, given $\eps > 0$ we may assume $i$ is taken large enough so that the flow can be written as a normal graph $v$ over $\Sigma$ with $\norm{v(\cdot, s)}_{C^{2,\alpha}} < \eps_0$ on $(s_i - \eta_i, s_i + \eta_i)$. By the interior estimates of Ecker--Huisken \cite{EckerHuisken} we may assume $\eta_i > 1$ when $i$ is sufficiently large. \par 
	Fix an $i$ such that all of the above is satisfied. \cref{eq-5-3} implies that
	\begin{align*}
		\norm{v(\cdot,s_i - \log(2))}_{W} \le \norm{v(\cdot,s_i)}_W + C_0\abs{E_{\rel}[\Sigma_{v(\cdot,s_i)},\Sigma]}^{1/2} < \delta.
	\end{align*}
	Applying the extension lemma \cite[Lemma 2.2]{Schulze}, we get that $v$ can be extended to a solution $v$ to \cref{nonlinear-problem} on $(s_i - 2\log(2),s_i + \log(2))$ with $\norm{v}_{C^{2,\alpha}} < \eps_1$. Iterating the above using the forward monotonicity formula \cref{forward-monotonicity} (note that $E_{\rel}$ in our case is negative and decreasing from $-\infty$, so $\abs{E_{\rel}}$ is increasing from $-\infty$), we get a solution $v$ to \cref{nonlinear-problem} on $(-\infty, s_i)$ with $\norm{v}_{C^{2,\alpha} }< \eps_1$. By \cref{integral-bound}, we have $v \in W(\Sigma \times (-\infty, s_i))$. Hence
	\begin{align*}
		\lim_{s \to -\infty} \norm{\frac{\partial v}{\partial s}(\cdot, s)}_W = 0.
	\end{align*}
	Arguing as before using interior Schauder estimates together with \cref{schauder-estimate} shows that after a suitable time translation, $v$ converges backwards to a (smooth) static solution $\Sigma'$ to the RMCF that is graphical over $\Sigma$. As $\mathcal{M}$ is trapped, unique continuation \cite[Theorem 1.4]{Bernstein} implies that $\Sigma' = \Sigma$, as desired.
  \end{proof}
\begin{rem}
	As a byproduct of \cref{lojasiewicz-relative-entropy}, we can determine the rate of decay of the flow to the static solution, similar to the ODE case. Indeed, the rate (in the rescaled setting) is exponential if \cref{lojasiewicz} holds and polynomial depending on $\gamma$ if \cref{better-lojasiewicz} holds.
\end{rem}
In $\R^3$ and $\R^4$, under suitable low entropy conditions, we can guarantee the smoothness of the self-expanders that arise as blow up limits of smooth MCFs (note that singular self-expanders do exist, but they do not arise as blow-up limits). Consequently \cref{main-cor-mcf} gives a complete picture of MCFs coming out of a generic cone of low entropy. 
\begin{cor}
\label{cor-r3}
Suppose $\mathcal{C} \subset \R^3$ is a generic cone with $\lambda[\mathcal{C}] < 2$, then any smooth MCF coming out of $\mathcal{C}$ is either a (smooth) self-expander or a tame ancient RMCF starting from a self-expander $\Sigma$ asymptotic to $\mathcal{C}$, as constructed in \cref{main-theorem}.
\end{cor}
\begin{proof}
	It suffices to establish that any blow up limit of such flow is automatically smooth. This argument is essentially carried out in \cite[Lemma 3.1]{CHHAncient} (cf. \cite[Lemma 4.1]{BWSharp}). Let $\Sigma$ be a subsequential limit from \cref{subsequent-convergence} and let $p \in \mathrm{sing}(\Sigma)$. By lower-semicontinuity of the entropy, any tangent cone $\nu$ at $p$ is a stationary cone with entropy less than 2. We claim that any such cone must be flat. Take an iterated tangent cone $\nu'$ at a singular point $q \in \supp \nu$ (that is not the vertex). $\nu'$ is then a 2-dimension stationary cone with entropy at most 2 that splits off a line, which we write as $\nu' = \nu'' \times \mu_\R$ for some 1-dimensional stationary cone $\nu''$. Any such $\nu''$ is a union of rays. Since the entropy of $\nu''$ is less than 2, $\nu''$ is either a flat line or a triple junction. However, as the MCF is smooth, it is cyclic as a mod 2 flat chain, which cannot encounter any singularity modeled on a triple junction by works of White \cite{WhiteCurrentsChains}. Hence $\nu''$ is a flat line, and $\nu'$ is flat. This implies $\nu$ is smooth away from the origin. Since $\nu$ is stationary and has entropy less than 2, $\supp \nu \cap \mathbb{S}^2$ is a closed geodesic of $\mathbb{S}^2$, which must be a multiplicity 1 great circle. This implies that $\nu$ is flat and $\Sigma$ is smooth by Allard regularity.
\end{proof}
\begin{cor}
\label{cor-r4}
Suppose $\mathcal{C} \subset \R^4$ is a generic cone with $\lambda[\mathcal{C}] < \frac{\pi}{2}$, then any smooth MCF coming out of $\mathcal{C}$ is either a (smooth) self-expander or a tame ancient RMCF starting from a self-expander $\Sigma$ asymptotic to $\mathcal{C}$, as constructed in \cref{main-theorem}.
\end{cor}
\begin{proof}
	The proof is similar except one uses \cite[Lemma 4.2]{CHHW} instead of \cite[Lemma 3.1]{CHHAncient}. Essentially, the same argument follows through until the conclusion $\supp \nu \cap \mathbb{S}^3$ is a closed smooth minimal surface in $\mathbb{S}^3$. It follows from the resolution of Willmore conjecture \cite{MarquesNeves} that $\supp \nu \cap \mathbb{S}^3$ must be an equatorial sphere as any other such minimal surface has (Gaussian) area ratio at least  $\frac{2\pi^2}{4\pi} = \frac{\pi}{2}$. Hence $\nu$ is flat and $\Sigma$ is smooth.
\end{proof}
\appendix

\section{The ODE Lemma}
\label{ode-appendix}
\begin{lem}[Lemma B.1 in \cite{ChoiMantoulidis}]
\label{ode-lemma}
	Suppose $x,y,z: (-\infty,0] \to [0,\infty)$ are absolutely continuous functions such that $x + y + z > 0$ and 
	\begin{align*}
		\liminf_{s \to -\infty} y(s) = 0.
	\end{align*}
	If there is $\eps > 0$ such that $x, y, z$ satisfy the following system of differential inequalities
	\begin{gather*}
		\abs{x'} \le \eps(x+y+z), \\
		y' + y \le \eps(x+z), \\
		z' - z \ge -\eps(x+y). 
	\end{gather*}
	Then there exists $\eps_0 > 0$ such that if $\eps \le \eps_0$, $y \le 2\eps(x+z)$. Moreover, either there exists $-\infty < s_* \le 0$ such that $z \le 8\eps x$ on $(-\infty,s_*)$ or $x \le c\eps z$ on $(-\infty,0]$ for some $c$ depending on $\eps_0$.
\end{lem}

{\footnotesize
	\bibliographystyle{alpha}
	\bibliography{lojasiewiczref.bib}}

\begin{thebibliography}{CHHW22}

\bibitem[ADS19]{ADS}
Sigurd Angenent, Panagiota Daskalopoulos, and Natasa Sesum.
\newblock Unique asymptotics of ancient convex mean curvature flow solutions.
\newblock {\em J. Differential Geom.}, 111(3):381 -- 455, 2019.

\bibitem[AIC95]{AngenentIlmanenChopp}
Sigurd~B. Angenent, Tom Ilmanen, and David~L. Chopp.
\newblock A computed example of nonuniqueness of mean curvature flow in
  {$\mathbb{R}^3$}.
\newblock {\em Commun. in Partial Differential Equations}, 20, 1995.

\bibitem[BCW]{BCW}
Jacob Bernstein, Letian Chen, and Lu~Wang.
\newblock Existence of {M}orse flow lines of the expander functional.
\newblock In preparation.

\bibitem[Ber20]{Bernstein}
Jacob Bernstein.
\newblock Asymptotic structure of almost eigenfunctions of drift laplacians on
  conical ends.
\newblock {\em Amer. J. Math.}, 142(6):1897--1929, 2020.

\bibitem[Bra78]{Brakke}
Kenneth~A. Brakke.
\newblock {\em The Motion of a Surface by its Mean Curvature}, volume~20 of
  {\em Mathematical Notes}.
\newblock Princeton University Press, Princeton, NJ, 1978.

\bibitem[BW16]{BWSharp}
Jacob Bernstein and Lu~Wang.
\newblock A sharp lower bound for the entropy of closed hypersurfaces up to
  dimension six.
\newblock {\em Invent. Math.}, 206:601--627, 2016.

\bibitem[BW17]{BWTopology}
Jacob Bernstein and Lu~Wang.
\newblock A topological property of asymptotically conical self-shrinkers of
  small entropy.
\newblock {\em Duke Math. J.}, 166(3):403--435, 2017.

\bibitem[BW18]{BWIntegerDegree}
Jacob Bernstein and Lu~Wang.
\newblock An integer degree for asymptotically conical self-expanders.
\newblock \url{https://arxiv.org/abs/1807.06494}, 2018.
\newblock Preprint.

\bibitem[BW21a]{BWSmoothCompactness}
Jacob Bernstein and Lu~Wang.
\newblock Smooth compactness for spaces of asymptotically conical
  self-expanders of mean curvature flow.
\newblock {\em Int. Math. Res. Not. IMRN}, 2021(12):9016--9044, 2021.

\bibitem[BW21b]{BWSpace}
Jacob Bernstein and Lu~Wang.
\newblock The space of asymptotically conical self-expanders of mean curvature
  flow.
\newblock {\em Math. Ann.}, 380:175--230, 2021.

\bibitem[BW22a]{BWMountainPass}
Jacob Bernstein and Lu~Wang.
\newblock A mountain-pass theorem for asymptotically conical self-expanders.
\newblock {\em Peking Math. J.}, 5:213--278, 2022.

\bibitem[BW22b]{BWRelativeEntropy}
Jacob Bernstein and Lu~Wang.
\newblock Relative expander entropy in the presence of a two-sided obstacle and
  applications.
\newblock {\em Adv. Math.}, 399(108284):1--48, 2022.

\bibitem[BW22c]{BWTopologicalUniqueness}
Jacob Bernstein and Lu~Wang.
\newblock Topological uniqueness for self-expanders of small entropy.
\newblock {\em Camb. J. Math.}, 10(4):785--833, 2022.

\bibitem[CCMS20]{CCMSGeneric}
Otis Chodosh, Kyeongsu Choi, Christos Mantoulidis, and Felix Schulze.
\newblock Mean curvature flow with generic initial data.
\newblock \url{https://arxiv.org/abs/2003.14344}, 2020.
\newblock Preprint.

\bibitem[Che22]{ChenII}
Letian Chen.
\newblock Rotational symmetry of solutions of mean curvature flow coming out of
  a double cone {II}.
\newblock \url{https://arxiv.org/abs/2201.11179}, 2022.
\newblock Preprint.

\bibitem[CHH22]{CHHAncient}
Kyeongsu Choi, Robert Haslhofer, and Or~Hershkovits.
\newblock Ancient low entropy flows, mean convex neighborhoods, and uniqueness.
\newblock {\em Acta. Math.}, 228(2):217--301, 2022.

\bibitem[CHHW22]{CHHW}
Kyeongsu Choi, Robert Haslhofer, Or~Hershkovits, and Brian White.
\newblock Ancient asymptotically cylindrical flows and applications.
\newblock {\em Invent. Math.}, 229:139--241, 2022.

\bibitem[CM12]{CMGeneric}
Tobias~Holck Colding and William~P. {Minicozzi II}.
\newblock Generic mean curvature flow {I}; generic singularities.
\newblock {\em Ann. of Math. (2)}, 175(2):755--833, 2012.

\bibitem[CM15]{CMUniqueness}
Tobias~Holck Colding and William~P. {Minicozzi II}.
\newblock Uniqueness of blowups and {Ł}ojasiewicz inequalities.
\newblock {\em Ann. of Math. (2)}, 182(1):221--285, 2015.

\bibitem[CM22]{ChoiMantoulidis}
Kyeongsu Choi and Christos Mantoulidis.
\newblock Ancient gradient flows of elliptic functionals and {M}orse index.
\newblock {\em Amer. J. Math.}, 144(2), 2022.

\bibitem[CS21]{ChodoshSchulze}
Otis Chodosh and Felix Schulze.
\newblock Uniqueness of asymptotically conical tangent flows.
\newblock {\em Duke Math. J.}, 170(16):3601--3657, 2021.

\bibitem[DH]{DuHaslhofer}
Wenkui Du and Robert Haslhofer.
\newblock Hearing the shape of ancient noncollapsed flows in {$\mathbb{R}^4$}.
\newblock {\em Comm. Pure Appl. Math.}
\newblock To Appear.

\bibitem[DS20]{DeruelleSchulze}
Alix Deruelle and Felix Schulze.
\newblock Generic uniqueness of expanders with vanishing relative entropy.
\newblock {\em Math. Ann.}, 377:1095--1127, 2020.

\bibitem[EH89]{EckerHuisken}
Klaus Ecker and Gerhard Huisken.
\newblock Mean curvature evolution of entire graphs.
\newblock {\em Ann. of Math. (2)}, 130(3):453--471, 1989.

\bibitem[FM20]{FeehanMaridakis}
Paul M.~N. Feehan and Manousos Maridakis.
\newblock Łojasiewicz–{S}imon gradient inequalities for analytic and
  {Morse–Bott} functions on {B}anach spaces.
\newblock {\em J. Reine Angew. Math.}, 2020(765):35--67, 2020.

\bibitem[Hel12]{Helmensdorfer}
Sebastian Helmensdorfer.
\newblock A model for the behavior of fluid droplets based on mean curvature
  flow.
\newblock {\em SIAM J. Math. Anal.}, 44(3):1359--1371, 2012.

\bibitem[Ilm94]{Ilmanen}
Tom Ilmanen.
\newblock {\em Elliptic Regularization and Partial Regularity for Motion by
  Mean Curvature}.
\newblock Number 520 in Mem. Amer. Math. Soc. American Mathematical Society,
  Providence, RI, 1994.

\bibitem[LSU68]{LSU}
O.A. Ladyzenskaja, V.~A. Solonnikov, and N.~N. Ural'ceva.
\newblock {\em Linear and quasi-linear equations of parabolic type}, volume~23
  of {\em Translations of Mathematical Monographs}.
\newblock American Mathematical Society, Providence, RI, 1968.

\bibitem[MN14]{MarquesNeves}
Fernando~C. Marques and André Neves.
\newblock Min-max theory and the {W}illmore conjecture.
\newblock {\em Ann. of Math. (2)}, 179(2):683--782, 2014.

\bibitem[PW]{ParkWang}
Jiewon Park and Lu~Wang.
\newblock Lojasiewicz inequality for relative expander entropy.
\newblock In preparation.

\bibitem[Sch14]{Schulze}
Felix Schulze.
\newblock Uniqueness of compact tangent flows in mean curvature flow.
\newblock {\em J. Reine Angew. Math.}, 690(2014):163--172, 2014.

\bibitem[Sim83]{SimonParabolic}
Leon Simon.
\newblock Asymptotics for a class of non-linear evolution equations, with
  applications to geometric problems.
\newblock {\em Ann. of Math. (2)}, 118(3):525--571, 1983.

\bibitem[Sim84]{SimonBook}
Leon Simon.
\newblock {\em Lectures on Geometric Measure Theory}.
\newblock Proceedings of the Centre of Mathematical Analysis, Australian
  National University. Australian National University, Centre of Mathematical
  Analysis, Canberra, 1984.

\bibitem[Sim96]{SimonETH}
Leon Simon.
\newblock {\em Theorems on Regularity and Singularity of Energy Minimizing
  Maps}.
\newblock Lectures in Mathematics. ETH Zürich. Birkhäuser, Basel, 1996.

\bibitem[Smo98]{Smoczyk}
Knut Smoczyk.
\newblock Starshaped hypersurfaces and the mean curvature flow.
\newblock {\em Manuscripta Math.}, 95:225--236, 1998.

\bibitem[SW]{SunWang}
Ao~Sun and Zhihan Wang.
\newblock Translating mean curvature flow with simple end.
\newblock In preparation.

\bibitem[Whi05]{WhiteRegularity}
Brian White.
\newblock A local regularity theorem for mean curvature flow.
\newblock {\em Ann. of Math. (2)}, 161(3):1487--1519, 2005.

\bibitem[Whi09]{WhiteCurrentsChains}
Brian White.
\newblock Currents and flat chains associated to varifolds, with an application
  to mean curvature flow.
\newblock {\em Duke Math. J.}, 148(1):41--62, 2009.

\end{thebibliography}

\end{document}